\documentclass[11pt]{article}
\usepackage{epigamath}

\setpapertype{A4}

\usepackage[english]{babel}

\usepackage[dvips]{graphicx}     

\title{\vspace{-0cm}Hilbert-Mumford stability on algebraic stacks and applications to $\mathcal{G}$-bundles on curves}
\titlemark{GIT-stability for stacks}
\author{\vspace{0cm}Jochen Heinloth}
\authoraddresses{
\authordata{Jochen Heinloth}{\firstname{Jochen} \lastname{Heinloth}\\
\institution{Universit\"at Duisburg--Essen, Fachbereich Mathematik, Universit\"atsstrasse 2, 45117 Essen, Germany}\\
\email{Jochen.Heinloth@uni-due.de}}}
\authormark{J. Heinloth}
\date{\vspace{-5ex}} 
\journal{\'Epijournal de G\'eom\'etrie Alg\'ebrique} 
\acceptation{Received by the Editors on October 6, 2016, and in
final form on July 23, 2017. \\ Accepted on December 12, 2017.}

\newcommand{\articlereference}{Volume 1 (2017), Article Nr. 11}

\usepackage[all]{xy}

\usepackage{MnSymbol}
\usepackage{stmaryrd}

\newtheorem{satz}{Satz}[section]
\newtheorem{proposition}[satz]{Proposition}
\newtheorem{theorem}[satz]{Theorem}
\newtheorem{definition}[satz]{Definition}
\newtheorem{lemma}[satz]{Lemma}

\newtheorem{corollary}[satz]{Corollary}

\newtheorem{remark}[satz]{Remark}

\newtheorem{example}[satz]{Example}

\newcommand{\tensor}{\otimes}

\newcommand{\map}[1]{\stackrel{#1}{\longrightarrow}}
\newcommand{\incl}[1]{\stackrel{#1}{\hookrightarrow}}

\newcommand{\un}[1]{\ensuremath{\protect\underline{#1}}}

\def\kbar{{\overline{k}}}

\def\GL{\textrm{GL}}

\def\SL{\textsf{SL}}

 \DeclareMathOperator{\Pic}{Pic}
 \DeclareMathOperator{\Bun}{Bun}
\DeclareMathOperator{\Gr}{Gr} \DeclareMathOperator{\GR}{GR}

\DeclareMathOperator{\Hom}{Hom}
\DeclareMathOperator{\cHom}{\mathcal{H}om}
\DeclareMathOperator{\Isom}{Isom} \DeclareMathOperator{\Map}{Map}
\DeclareMathOperator{\Mor}{Mor} \DeclareMathOperator{\Spec}{Spec}
\DeclareMathOperator{\Proj}{Proj} 

 \DeclareMathOperator{\Aut}{Aut}
 \DeclareMathOperator{\End}{End}
\DeclareMathOperator{\cEnd}{\mathcal{E}nd}
 \DeclareMathOperator{\Lie}{Lie}

\DeclareMathOperator{\supp}{supp}
 
 \DeclareMathOperator{\tr}{tr}
 \DeclareMathOperator{\rk}{rk}
 
 \DeclareMathOperator{\Stab}{Stab}
\DeclareMathOperator{\Gal}{Gal} \DeclareMathOperator{\Ram}{Ram}
\DeclareMathOperator{\Res}{Res} 
\DeclareMathOperator{\Ad}{Ad} \DeclareMathOperator{\Centr}{Centr}
\DeclareMathOperator{\Rees}{Rees}
\DeclareMathOperator{\Chain}{Chain}
\DeclareMathOperator{\Transp}{Transp} \DeclareMathOperator{\wt}{wt}
\DeclareMathOperator{\glue}{glue}

\def\Bl{\mathrm{Bl}}

\DeclareMathOperator{\im}{Im} 
\DeclareMathOperator{\act}{act} \DeclareMathOperator{\ad}{ad}
\DeclareMathOperator{\conj}{conj} \DeclareMathOperator{\Char}{char}
\DeclareMathOperator{\ST}{ST}

\def\gr{\mathrm{gr}}
\def\forget{\mathit{forget}}
\def\pr{\mathit{pr}}

\def\1halb{\frac{1}{2}}

\def\pbl{[\![}
\def\pbr{]\!]}

\def\sxymat{\xymatrix@C=1.5ex@R=0.8ex}
\def\grp{$\xymatrix{ R\times_{X}R  \ar[r]^-{\mu} & R \ar@<1ex>[r]^-{s}\ar@<-1ex>[r]_-{t} & X}$}
\def\dar{\ar@<-0.5ex>[r]\ar@<0.5ex>[r]}
\def\tar{\ar[r]\ar@<1ex>[r]\ar@<-1ex>[r]}
\newcommand{\dmap}[2]{\ar@<-0.5ex>[r]_-{#2}\ar@<0.5ex>[r]^-{#1}}
\newcommand{\dotarrow}[2]{\xymatrix{{#1}\ar@{..>}[r]&{#2}}}
\def\cart{\ar@{}[dr]|{\square}}

\def\cB{\mathcal{B}}

\def\cE{\mathcal{E}}
\def\cF{\mathcal{F}}
\def\cG{\mathcal{G}}

\def\cL{\ensuremath{\mathcal{L}}}
\def\cM{\mathcal{M}}

\def\cO{\mathcal{O}}
\def\cP{\mathcal{P}}
\def\cQ{\mathcal{Q}}

\def\cT{\mathcal{T}}
\def\cU{\mathcal{U}}
\def\cV{\mathcal{V}}

\def\cg{\mathfrak{g}}

\def\cu{\mathfrak{u}}

\def\bA{{\mathbb A}}

\def\bD{{\mathbb D}}

\def\bG{{\mathbb G}}

\def\bN{{\mathbb N}}

\def\bP{{\mathbb P}}
\def\bQ{{\mathbb Q}}
\def\bR{{\mathbb R}}

\def\bZ{{\mathbb Z}}

\def\oST{\overline{\ST}}

\DeclareMathOperator{\id}{id}

\begin{document}


\maketitle



\begin{prelims}


\def\abstractname{Abstract}
\abstract{In these notes we reformulate the classical
Hilbert-Mumford criterion for GIT stability in terms of algebraic
stacks, this was independently done by Halpern-Leinster \cite{DHL}.
We also give a geometric condition that guarantees the existence of
separated coarse moduli spaces for the substack of stable objects.
This is then applied to construct coarse moduli spaces for torsors
under parahoric group schemes over curves.}



\vspace{0.05cm}

\languagesection{Fran\c{c}ais}{%

\textbf{Titre. Stabilit\'e de Hilbert-Mumford sur les champs
alg\'ebriques et applications aux $\mathcal{G}$-fibr\'es sur les
courbes} \commentskip \textbf{R\'esum\'e.} Dans ces notes, nous
reformulons le crit\`ere classique de stabilit\'e de Hilbert-Mumford pour la
th\'eorie g\'eom\'etrique des invariants en termes de champs
alg\'ebriques ; cela a \'et\'e fait ind\'epen\-damment par
Halpern-Leinster \cite{DHL}. Nous donnons \'egalement une condition
g\'eom\'etrique qui garantit l'existence d'espaces de modules
grossiers pour le sous-champ des objets stables. Nous l'appliquons ensuite pour construire des espaces de modules grossiers pour des
torseurs sous des sch\'emas en groupes parahoriques sur des
courbes.}

\end{prelims}


\newpage

\setcounter{tocdepth}{1}
\tableofcontents

\section*{Introduction}

The aim of these notes is to reformulate the Hilbert-Mumford
criterion from geometric invariant theory (GIT) in terms of
algebraic stacks (Definition \ref{Def:Lstab}) and use it to give an
existence result for separated coarse moduli spaces.

Our original motivation was that for various moduli problems one has
been able to guess stability criteria which have then been shown to
coincide with stability conditions imposed by GIT constructions of
the moduli stacks. It seemed strange to me that in these
constructions it is often not too difficult to find a stability
criteria by educated guessing, however, in order to obtain coarse
moduli spaces one then has to prove that the guess agrees with the
Hilbert-Mumford criterion from GIT, which often turns out to be a
difficult and lengthy task.

Many aspects of GIT have of course been reformulated in terms of
stacks by Alper \cite{Alper}. Also Iwanari \cite{Iwanari} gave a
clear picture for pre-stable points on stacks and constructed
possibly non-separated coarse moduli spaces. The analog of the
numerical Hilbert-Mumford criterion has been used implicitly in many
places by several authors. Most recently Halpern-Leinster \cite{DHL}
independently gave a formulation, very close to ours and applied it
to construct analogs of the Harder--Narasimhan stratification for
moduli problems under a condition he calls $\Theta$-reductivity.

Our main aim is to give a criterion that guarantees that the stable
points form a separated substack (Proposition \ref{Prop:sep}). Once
this is available, one can apply general results (e.g., the theorems
of Keel and Mori \cite{KM} and Alper, Hall and Rydh
\cite{AlperHallRydh}) to obtain separated coarse moduli spaces
(Proposition \ref{Prop:coarse}). As side effect, we hope that our
formulation may serve as an introduction to the beautiful picture
developed \cite{DHL}.

The guiding examples which also provide our main applications are
moduli stacks of torsors under parahoric group schemes. Using our
method we find a stability criterion for such torsors on curves and
construct separated coarse moduli spaces of stable torsors (Theorem
\ref{Thm:parahoric_moduli}). Previously such moduli spaces had been
constructed for generically trivial group schemes in characteristic
0 by Balaji and Seshadri \cite{BS}, who obtain that in these cases
the spaces are schemes. Also in any characteristic the special case
of moduli of parabolic bundles has been constructed in \cite{HS},
but it seems that coarse moduli spaces for twisted groups had not
been constructed before.

This problem was the starting point for the current article, because
in \cite{HS} most of the technical problems arose in the
construction of coarse moduli spaces by GIT that were needed to
prove cohomological purity results for the moduli stack. The results
of this article allow to bypass this issue and apply to the larger
class of parahoric groups.

The structure of the article is as follows. In Section 1 we state
the stability criterion depending on a line bundle $\cL$ on an
algebraic stack $\cM$. As a consistency check we then show that this
coincides with the Hilbert Mumford criterion for global quotient
stacks (Proposition \ref{Prop:GITstability}). To illustrate the
method we then consider some classical moduli problems and show how
Ramanathan's criterion for stability of  $G$-bundles on curves can
be derived form our criterion rather easily. The same argument
applies to related moduli spaces, as the moduli of chains or pairs.

In Section 2 we formulate the numerical condition on the pair
$(\cM,\cL)$ implying that the stable points form a separated
substack (Proposition \ref{Prop:sep}) and derive an existence result
for coarse moduli spaces (Proposition \ref{Prop:coarse}). Again, as
an illustration we check that this criterion is satisfied in for
GIT-quotient stacks and for $G$-bundles on curves.

Finally in Section 3 we apply the method to the moduli stack of
torsors under a parahoric group scheme on a curve. We construct
coarse moduli spaces for the substack of stable points of these
stacks. For this we also need to prove some of the basic results
concerning stability of parahoric group schemes that could be of
independent interest.

\bigskip
\noindent {\bf Acknowledgments:} This note grew out of a talk given
in Chennai in 2015. The comments after the lecture encouraged me to
finally revise of an old sketch that had been on my desk for a very
long time. I am grateful for this encouragement and opportunity. I
thank M. Olsson for pointing out the reference \cite{DHL} before it
appeared and J. Alper, P. Boalch and the referees for many helpful
comments and suggestions. While working on this problem, discussions
with V.\ Balaji, N.\ Hoffmann, J.\ Martens, A.\ Schmitt have been
essential for me.

\section{The Hilbert-Mumford criterion in terms of stacks}
Throughout we will work over a fixed base field $k$. The letter
$\cM$ will denote an algebraic stack over $k$, which is locally of
finite type over $k$ and we will always assume that the diagonal
$\Delta\colon \cM \to \cM \times \cM$ is quasi-affine, as this
implies that our stack is a stack for the fpqc topology
(\cite[Corollaire 10.7]{LMB}).

We have two guiding examples in mind: First, global quotient stacks
$[X/G]$ where $X$ is a proper scheme and $G$ is an affine algebraic
group acting on $X$ and second, the stack $\Bun_G$ of $G$-bundles on
a smooth projective curve $C$ for a semi-simple group $G$ over $k$.

\subsection{Motivation: The classical Hilbert-Mumford criterion}
As the numerical criterion for stability from geometric invariant
theory \cite[Theorem 2.1]{GIT} serves as a guideline we start by
recalling this briefly. To state it and in order to fix our sign
conventions we need to recall the definition of weights of
$\bG_m$-equivariant line bundles:
\subsubsection{Weights of equivariant line bundles}
As usual we denote the multiplicative group scheme by $\bG_m:= \Spec
k[t,t^{-1}]$ and the affine line by $\bA^1:=\Spec k[x]$. The
standard action $$\act\colon \bG_m \times \bA^1 \to \bA^1$$ is given
by $t.x := tx$, i.e. on the level of rings  $\act^{\#}\colon k[x]
\to k[x,t,t^{-1}]$ is given by $\act^{\#}(x)=tx$. We write
$[\bA^1/\bG_m]$ for the quotient stack defined by this action. This
stack is called $\Theta$ in \cite{DHL}.

By definition $\bG_m$-equivariant line bundles $\cL$ on $\bA^1$ are
the same as line bundles on $[\bA^1/\bG_m]$. For any line bundle
$\cL$ on $\bA^1$ the global sections are $H^0(\bA^1,\cL)=k[x] \cdot
e$ for some section $e$ which is unique up to a scalar multiple.

Thus an equivariant line bundle $\cL$ on $\bA^1$ defines an integer
$d\in \bZ$ by $\act^{\#}(e)= t^d e$ called the weight of $\cL$ and
we will denote it as
$$ \wt(\cL) := d.$$
In particular we find that:
$$H^0([\bA^1/\bG_m],\cL)= H^0(\bA^1,\cL)^{\bG_m} = \left\{ \begin{array}{ll} k\cdot x^{d} e & \text{if } \wt(\cL)=d \leq 0 \\ 0 & \text{if } \wt(\cL)=d >0. \end{array} \right.$$
To compare the sign conventions in different articles the above
equation is the one to keep in mind, because Mumford's construction
of quotients uses invariant sections of line bundles.
Similarly, a $\bG_m$-equivariant line bundle $\cL$ on $\Spec k$ is
given by a morphism of underlying modules $\act^{\#}\colon \cL \to
\cL \tensor k[t,t^{-1}]$  with $\act^{\#}(e) = t^d e$ for some $d\in
\bZ$. The integer $d$ is again denoted by:
$$\wt_{\bG_m}(\cL)\coloneq d.$$
\subsubsection{Mumford's notion of stability}
Mumford considers a projective scheme $X$, equipped with an action
of a reductive group $G$ and a $G$-linearized ample bundle $\cL$ on
$X$. For any $x\in X(k)$ and any cocharacter $\lambda\colon \bG_m
\to G$ Mumford defines $\mu^{\cL}(x,\lambda)\in \bZ$ as follows. The
action of $\bG_m$ on $x$ defines a morphism $\lambda.x\colon \bG_m
\to X$ which extends to an equivariant morphism $f_{x,\lambda}\colon
\bA^1 \to X$ because $X$ is projective. He defines
$$ \mu^{\cL}(x,\lambda)\coloneq - \wt(f_{x,\lambda}^* \cL).$$

The criterion \cite[Theorem 2.1]{GIT} then reads as follows:

A geometric point $x\in X(\kbar)$ is {\em stable} if and only if the
stabilizer of $x$ in $G$ is finite and for all $\lambda\colon \bG_m
\to G$ we have
$$\mu^L(x,\lambda) >0, \text{ equivalently } \wt(f_{x,\lambda}^*\cL) <0.$$
\subsection{$\cL$-stability on algebraic stacks}
It is easy to reformulate the numerical Hilbert--Mumford criterion in terms of stacks, once we fix some notation. The quotient stack $[\bA^1/\bG_m]$ has two geometric points $1$ and $0$ which are the images of the points of the same name in $\bA^1$. 
For any algebraic stack $\cM$ and $f\colon [\bA^1/\bG_m]\to \cM$ we
will write $f(0),f(1) \in \cM(k)$ for the points given by the images
of $0,1\in \bA^1(k)$.

\begin{definition}[Very close degenerations]
{\rm Let $\cM$ be an algebraic stack over $k$ and $x\in \cM(K)$ a
geometric point for some algebraically closed field $K/k$.

A {\em very close degeneration of $x$} is a morphism $f\colon
[\bA^1_K/\bG_{m,K}] \to \cM$ with $f(1)\cong x$ and $f(0)\not\cong
x$.}
\end{definition}
Very close degenerations have been used under different names, e.g.
in the context of $K$-stability these are often called
test-configurations. Our terminology should only emphasize that
$f(0)$ is an object that lies in  the closure of a $K$ point of
$\cM_K$, which only happens for stacks and orbit spaces, but if
$X=\cM$ is a scheme, then there are no very close degenerations.

\begin{definition}[$\cL$-stability]\label{Def:Lstab}
{\rm Let $\cM$ be an algebraic stack over $k$, locally of finite
type with affine diagonal and $\cL$ a line bundle on $\cM$. A
geometric point $x\in \cM(K)$ is called {\em $\cL$-stable} if
\begin{enumerate}
\item[\rm (1)] for all very close degenerations $f\colon [\bA^1_K/\bG_{m,K}] \to \cM$ of $x$ we have $$\wt(f^*\cL) <0$$
and
\item[\rm (2)] $\dim_K(\Aut_{\cM}(x)) =0$.
\end{enumerate}}
\end{definition}

\begin{remark}\label{Rem:L-stable}
{\rm \begin{enumerate}
\item[]
\item[\rm (1)] We can also introduce the notion of {\em $\cL$-semistable} points, by requiring only $\leq$ in (1) and dropping condition (2). 
\item[\rm (2)] The notion admits several natural extensions: Since the weight of line bundles extends to a elements of the groups $$\Pic(B\bG_m)\tensor_{\bZ} \bR\cong \bR,$$ the above definition can also be applied if $\cL\in
\Pic(\cM)\tensor_{\bZ}\bR$. This is often convenient for classical
notions of stability that depend on real parameters.

In \cite{DHL} Halpern-Leinster uses a cohomology class $\alpha \in
H^2(\cM_{\kbar},\bQ_\ell)$ instead of a line bundle. For
$\alpha=c_1(\cL)$ this gives the same condition. As the condition
given above is a numerical one, this is sometimes more convenient
and we will refer to it as $\alpha$-stability.
\item[\rm (3)] As the above condition is numerical and $H^{2d}([\bA^1/\bG_m],\bQ_\ell)=\bQ_\ell$ for all $d>0$
one can generalize the notion further.
\item[\rm (4)] For the above definition to be nonempty it is convenient to assume that the automorphism groups of generic objects are finite, so that (2) can be fulfilled. This appears to rule out the example of vector bundles on curves, but fortunately there is a quite general method available to rigidify the problem, i.e., to divide out generic automorphism groups (\cite[Theorem 5.1.5]{ACV},\cite[Appendix C]{AGV}).
\item[\rm (5)] General representability results of \cite[Theorem 1.6]{HLP}  imply that under our conditions on $\cM$ the stack of morphisms $\Mor([\bA^1/\bG_m],\cM)$ is again an algebraic stack, locally of finite type. As weights of $\bG_m$-actions on line bundles are locally constant in families this implies that $\cL$-stability is preserved under extension of algebraically closed fields $L/K$. In particular, if we define an $S$-valued point $\cM(S)$ to be $\cL$-stable if the corresponding objects are stable for all geometric points of $S$, we get a notion that is preserved under pull-back and therefore defines an abstract substack $\cM^s\subset \cM$. We will see that in may cases $\cL$-stability turns out to be an open condition and then $\cM^s$ is again an algebraic stack, but this is not true for arbitrary $\cM,\cL$.
\end{enumerate}}
\end{remark}

\noindent {\bf Notation.}
    Given a line bundle $\cL$ on $\cM$ and $x\in \cM(K)$ we will denote by $\wt_x(\cL)$ the  homomorphism $$\wt_x(\cL) \colon X_*(\Aut_{\cM}(x))=\Hom(\bG_m,\Aut_{\cM}(x)) \to \bZ$$
    which maps $\lambda\colon \bG_m \to \Aut_{\cM}(x)$ to $\wt_{\bG_m}((x,\lambda)^*\cL)$, where
    $(x,\lambda)\colon [\Spec K/\bG_m] \to \cM$ is the morphism defined by $x$ and $\lambda$.

\smallskip

\begin{example} {\rm A toy example illustrating the criterion is given by the anti-diagonal action $\bG_m\times \bA^2 \to \bA^2$ defined as $t.(x,y)\coloneq (tx,t^{-1}y)$.
The only fixed point of this action is the origin $0$. The quotient
$(\bA^2-\{0\})/\bG_m $ is the affine line with a doubled origin, the
first example of a non-separated scheme.

Since the latter space is a scheme none of its points admits very
close degenerations. This changes if we look at the full quotient
$[\bA^2/\bG_m]$, which contains the additional point
$[(0,0)/\bG_m]$. The inclusions of the coordinate axes
$\iota_{x},\iota_y \colon \bA^1 \to \bA^2$ define very close
degenerations of the points $(0,1)$ and $(1,0)$ and it will turn out
(Lemma \ref{vcd_quotientstacks}) that these constitute essentially
the only very close degenerations in this stack.

A line bundle on $[\bA^2/\bG_m]$ is an equivariant line bundle on
$\bA^2$. Since line bundles on $\bA^2$ are trivial, all equivariant
line bundles are given by a character $()^d\colon \bG_m \to \bG_m$
and we find $\wt_{0}\colon \Pic([\bA^2/\bG_m])\cong \bZ$. Moreover
for the corresponding line bundle $\cL_d$ the weights
 $\wt(\iota_x^*\cL_d) = d, \wt(\iota_y^*\cL_d) = -d$ so for each $d\neq 0$ only one of the points $(1,0)$ and $(0,1)$ can be $\cL_d$-stable.

For $d=0$ the points $(1,0),(0,0),(0,1)$ are all semistable and
these points would be identified in the GIT quotient.}
\end{example}

\subsection{Determining very close degenerations}
To apply the definition of $\cL$-stability, one needs to classify
all very close degenerations. The next lemma shows that these can be
described by deformation theory of objects $x$ that admit
non-constant morphisms $\bG_m \to \Aut_{\cM}(x)$. Let us fix our
notation for formal discs:
$$ \bD := \Spec k\pbl t\pbr, \mathring{\bD}:= \Spec k(\!(t)\!).$$

\begin{lemma}\label{Lem:vcdeformation} Let $\cM$ be an algebraic stack locally of finite type over $k=\kbar$ with quasi-affine diagonal.
\begin{enumerate}
\item[\rm (1)] For any very close degeneration $f\colon [\bA^1/\bG_m] \to \cM$ the induced morphism
$$ \lambda_f \colon \bG_m =\Aut_{[\bA^1/\bG_m]}(0) \to \Aut_{\cM}(f(0))$$
is nontrivial.
\item[\rm (2)] The restriction functor $\cM([\bA^1/\bG_m]) \to \varprojlim \cM([\Spec(k[x]/x^n)/\bG_m])$ is an equivalence of categories.
\end{enumerate}
\end{lemma}
\begin{proof}
For part (2) elegant proofs have been given independently by Alper,
Hall and Rydh \cite[Corollary 3.6]{AlperHallRydh} and Bhatt,
Halpern-Leistner \cite[Remark 8.3]{BHL}. Both proofs rely on Tannaka
duality, one uses an underived version the other a derived version.
As the statement produces the point $f(1)$ out of a formal datum let
us explain  briefly why this is possible: The composition
$$\phi\colon \bD =\Spec k\pbl  x\pbr  \to \bA^1 \to [\bA^1/\bG_m]$$
is faithfully flat, because both morphisms are flat and the map is
surjective, because both points $1,0$ of $[\bA^1/\bG_m]$ are in the
image.

By our assumptions $\cM$ is a stack for the fpqc topology
(\cite[Corollaire 10.7]{LMB}) we therefore see that
$\cM([\bA^1/\bG_m])$ can be described as objects in $\cM(k\pbl
x\pbr)$ together with a descent datum with respect to $\phi$.

Moreover, the canonical map $\cM(\bD) \map{\cong} \varprojlim
\cM(k[x]/(x^n))$ is an equivalence of categories: This follows for
example, because the statement holds for schemes and choosing a
smooth presentation $X \to \cM$ one can reduce to this statement,
\cite[Tag 07X8]{Stacks}.

In particular this explains already that an element of $ \varprojlim
\cM([\Spec(k[x]/x^n)/\bG_m])$ will produce a $k\pbl x\pbr$-point of
$\cM$. The problem now lies in constructing a descent datum for this
morphism, as $$\bD  \times_{[\bA^1/\bG_m]} \bD = \Spec (k\pbl x\pbr
\tensor_{k[x]} k[x,t,t^{-1}] \tensor_{k[y]} k\pbl y\pbr )$$ where
the last tensor product is taken via $y=xt$. The ring on the right
hand side is not complete and the formal descent data coming form an
element in $\varprojlim \cM([\Spec(k[x]/x^n)/\bG_m])$ only seems to
induce a descent datum on the completion of the above ring.

Here the Tannakian argument greatly simplifies the problem, as it
gives a concise way to capture the information that a $\bG_m$ action
induces a grading and therefore allows to pass from power series to
polynomials.

Let us deduce (1). First note that this holds automatically if $\cM$
is a scheme, because then $f(1)$ is a closed point and $f(0)$ lies
in the closure of $f(1)$.

In general choose a smooth presentation $p\colon X \to \cM$. If
$\lambda_f$ is trivial, we can lift the morphism $f|_{0}\colon
[0/\bG_m] \to \Spec k \to \cM$ to $\tilde{f}_0 \colon [0/\bG_m] \to
X$. Since $p$ is smooth, we can inductively lift this morphism to
obtain an element in $ \varprojlim X([\Spec(k[x]/x^n)/\bG_m])$. Thus
we reduced (1) to the case $\cM=X$.
\hfill $\Box$
\end{proof}

Using this lemma, we can compare $\cL$-stability to classical
notions of stability, simply by first identifying objects for which
the automorphism group contains $\bG_m$ and then studying their
deformations. The next subsections illustrate this procedure in
examples.
\subsection{The example of GIT-quotients}\label{Sec:GIT-quot}
Let $X$ be a projective variety equipped with the action of a
reductive group $G$ and a $G$-linearized line bundle $\cL$. Again
bundles on the quotient stack $[X/G]$ are the same as equivariant
bundles on $X$, so we will  alternatively view $\cL$ as a line
bundle on $[X/G]$.

Let us fix some standard notation: Given $\lambda \colon \bG_m \to
G$ we will denote by $P_\lambda$ the corresponding parabolic
subgroup, $U_\lambda$ its unipotent radical and $L_\lambda$ the
corresponding Levi subgroup, i.e.,
\begin{align*}
P_\lambda(R) &= \{ g\in G(R) | \lim_{t\to 0} \lambda(t)g\lambda(t^{-1}) \text{ exists}\},\\
U_\lambda(R) &= \{ g\in G(R) | \lim_{t\to 0} \lambda(t)g\lambda(t^{-1})=1 \},\\
L_\lambda(R)&= \{ g\in G(R) | \lambda(t)g\lambda(t^{-1}) =g \}\text{
i.e., } L_\lambda=\Centr_G(\lambda).
\end{align*}

To compare $\cL$-stability on $[X/G]$ to GIT-stability on $X$ we
first observe that the test objects appearing in the conditions
coincide:
\begin{lemma}\label{vcd_quotientstacks}
For any cocharacter $\lambda\colon \bG_m \to G$ and any geometric
point $x\in X(K)$ that is not a fixed point of $\lambda$ the
equivariant map $f_{\lambda,x}\colon \bA^1_K \to X$ defines a very
close degeneration $\overline{f}_{\lambda,x}\colon
[\bA^1_K/\bG_{m,K}] \to [X/G]$. Moreover, any very close
degeneration in the stack  $[X/G]$ is of the form
$\overline{f}_{\lambda,x}$ for some $x,\lambda$.
\end{lemma}
\begin{proof}
Since $f_{\lambda,x}(0)$ is a fixed point of $\lambda$ and $x$ is
not, we have $\overline{f}(0)\not\cong \overline{f}(1)$, thus
$\overline{f}$ is a very close degeneration.

Conversely let $f\colon [\bA^1/\bG_m] \to [X/G]$ be any very close
degeneration. We need to find a $\bG_m$ equivariant morphism
$$\xymatrix{
{\bA^1}\ar[d]\ar@{-->}[r]^{\tilde{f}} & X\ar[d]^\pi \\
     [\bA^1/\bG_m]\ar[r]^{f} & [X/G].
}$$ Since $\pi\colon X \to [X/G]$ is a $G$-bundle, the pull-back
$p\colon X \times_{[X/G]} [\bA^1/\bG_m] \to [\bA^1/\bG_m]$ is a
$G$-bundle on $[\bA^1/\bG_m]$. To find $\tilde{f}$ is equivalent to
finding a $\bG_m$-equivariant section of this bundle. This will
follow from the known classification of $\bG_m$-equivariant
$G$-bundles on the affine line:
\begin{lemma}[$G$-bundles on ${[\bA^1/\bG_m]}$]\label{equivG}
Let $G$ be a reductive group and $\cP$ a $G$-bundle on
$[\bA^1/\bG_m]$. Denote by $\cP_0$ the fiber of $\cP$ over $0\in
\bA^1$.
\begin{enumerate}
    \item[\rm (1)] If there exists $x_0\in \cP_0(k)$ (e.g. this holds if $k=\kbar$). Then there exists a cocharacter $\lambda \colon \bG_m \to G$, unique up to conjugation and an isomorphism of $G$-bundles $$\cP \cong [\big(\bA^1 \times G\big)/\big(\bG_m,(\act,\lambda)\big)].$$
    Moreover, $\cP$ has a canonical reduction to $P_\lambda$.
    \item[\rm (2)] Let $G_0\coloneq \Aut_G(P_0)$ and $\lambda \colon \bG_m \to G_0$ the cocharacter  defined by $\cP|_{[0/\bG_m]}$. Consider the $G_0$ bundle  $\cP^{G_0}\coloneq \Isom_G(\cP,\cP_0)$ on $[\bA^1/\bG_m]$. Then $$\cP^{G_0} \cong [\big(\bA^1 \times G_0\big) /\big(\bG_m,(\act,\lambda)\big)],$$ i.e.
    $\cP \cong \Isom_{G_0}([\big(\bA^1 \times G_0\big)/\big(\bG_m,(\act,\lambda)\big)],\cP_0)$. Moreover, $\cP^{G_0}$ has a canonical reduction to  $P_{0,\lambda}\subset G_0$.
\end{enumerate}
\end{lemma}
For vector bundles this result is \cite[Theorem 1.1]{Asok} where
some history is given. The general case can be deduced from this
using the Tannaka formalism. As we will need a variant of the
statement later, we give a slightly different argument.
\begin{proof}[Proof of Lemma \ref{equivG}]
    The second part follows from the first, as the $G_0$-bundle $\cP^{G_0}_0=\Isom_G(\cP_0,\cP_0)$ has a canonical point $\id$. We added (2), because it gives an intrinsic statement, independent of choices.

    To prove (1) note that $x_0$ defines an isomorphism $G \map{\cong} \Aut_G(P_0)$ and a section $[\Spec k/\bG_m]\to \cP|_{[0/\bG_m]}$. This induces a section $[\Spec k/\bG_m]\to \cP/P_{\lambda}$. As the map $\pi\colon \cP/P_{\lambda}\to [\bA^1/\bG_m]$ is smooth, any section can be lifted infinitesimally to $\Spec k[x]/x^n$ for all $n\geq 0$. Inductively the obstruction to the existence of a $\bG_m$-equivariant section is an element in $H^1(B\bG_m, T_{(\cP/P_{\lambda})/\bA^1,0}\tensor (x^{n-1})/(x^n)) = 0$ and the choices of such liftings form a torsor under $H^0(B\bG_m,T_{ \cP/P_{\lambda}/\bA^1,0} \tensor (x^{n-1})/(x^n))$. Now by construction $\bG_m$ acts with negative weight on $T_{ \cP/P_{\lambda},0}= \Lie(G)/\Lie(P_{\lambda})$ and it also acts with negative weight on the cotangent space $(x)/(x^2)$, so there exists a canonical $\bG_m$ equivariant reduction $\cP_\lambda$ of $\cP$ to $P_\lambda$.

    Similarly, the vanishing of $H^1$ implies that we can also find a compatible family of $\lambda$-equivariant sections $[(\Spec k[t]/t^n)/\bG_m] \to \cP$ and by Lemma \ref{Lem:vcdeformation} this defines a section over $[\bA^1/\bG_m]$, i.e. a morphism of $G$-bundles $[(\bA^1 \times G)/\bG_m,\lambda] \to \cP$.

    Here, we could alternatively see this explicitly as follows: Let us consider $\pi\colon\cP_\lambda \to \bA^1$ as $\bG_m$ equivariant $P_\lambda$ bundle on $\bA^1$ and consider the twisted action
    $$\star\colon \bG_m \times \cP_{\lambda} \to \cP_{\lambda}$$ given by $t \star p \coloneq (t.p) \cdot \lambda(t^{-1})$. Our point $x_0$ is a fixed point for this action by construction, as we used it to identify $P_0$ with $G$ and $e\in G$ is a fixed point for the conjugation $\lambda(t)\cdot \un{\quad} \cdot \lambda(t^{-1})$. Moreover $\lambda$ acts with non-negative weights on $\Lie(P_\lambda)$ and also on $T_{\bA^1} =((x)/(x^2))^{\vee}$. Therefore Bia\l ynicki-Birula decomposition \cite{Hesselink} implies that there exists a point $x_1\in \pi^{-1}(1)$ such that $\lim_{t\to 0} t\star x_1 = x_0$. This defines a $\lambda$-equivariant section $\bA^1 \to \cP$.
\hfill $\Box$
\end{proof}
This also completes the proof of Lemma \ref{vcd_quotientstacks}.
\hfill $\Box$
\end{proof}

\begin{proposition}\label{Prop:GITstability}
Let $X,G,\cL$ be a projective $G$-scheme together with a
$G$-linearized line bundle $\cL$. A point $x\in X(\kbar)$ is
GIT-stable with respect to $\cL$ if and only if the induced point
$\overline{x} \in [X/G](\kbar)$ is $\cL$-stable.
\end{proposition}
\begin{proof}
As $X$ is projective, given $x\in X(\kbar)$ and a one parameter
subgroup $\lambda \colon \bG_m \to G$ we obtain an equivariant map
$f_{\lambda,x}\colon  \bA^1 \to X$ and thus a morphism $\overline{f}
\colon [\bA^1/\bG_m] \to X$. By Lemma \ref{vcd_quotientstacks} all
very close degenerations arise in this way. As
$\wt(\overline{f}^*\cL)= \wt(f^*\cL)$ we therefore find that $x$
satisfies the Hilbert-Mumford criterion for stability if and only if
it is $\cL$-stable.
\hfill $\Box$
\end{proof}

\subsection{Stability of vector bundles on curves}
We want to show how the classical notion of stability for
$G$-bundles arises as $\cL$-stability. For the sake of clarity we
include the case of vector bundles first. Let $C$ be a smooth,
projective, geometrically connected curve over $k$ and denote by
$\Bun_n^d$ the stack of vector bundles of rank $n$ and degree $d$ on
$C$.

\subsubsection{The line bundle}
A natural line bundle on $\Bun_n^d$ is given by the determinant of
cohomology $\cL_{\det}$, i.e., for any vector bundle $\cE$ on $C$ we
have $\cL_{\det,\cE}:= \det(H^*(C,\cEnd(\cE)))^{-1},$ and more
generally for any $f\colon T\to \Bun_n^d$ corresponding to family
$\cE$ on $C\times T$ one defines
$$f^*\cL_{\det} := \big(\det \bR \pr_{T,*} \cEnd(\cE)\big)^\vee.$$

\begin{remark}\label{rigidify}
{\rm    Since any vector bundle admits $\bG_m$ as central automorphisms, to apply our criterion we need to pass to the rigidified stack $\overline{\Bun}_n^d:=\Bun_n^d\!\!\fatslash\, \bG_m$ obtained by dividing all automorphism groups by $\bG_m$ (\cite[Theorem 5.1.5]{ACV}). To obtain a line bundle on this stack, we need a line bundle on $\Bun_n^d$ on which the central $\bG_m$-automorphisms act trivially. This is the reason why we use the determinant of $H^*(C, \cEnd(\cE))$ instead of $H^*(C,\cE)$. It is known that up to multiples and bundles pulled back from the Picard variety this is the only such line bundle on $\Bun_n^d$ (see e.g. \cite{BiswasHoffmann} which also gives some history on the Picard group of $\Bun_n^d$).}
\end{remark}

\subsubsection{Very close degenerations of bundles}
To classify maps $f\colon [\bA^1/\bG_m] \to \Bun_n^d$ we have as in
Lemma \ref{equivG}.

\begin{lemma}
There is an equivalence
$$ \Map([\bA^1/\bG_m],\Bun_n^d) \cong \langle (\cE,\cE^i)_{i\in\bZ} | \cE^{i} \subset \cE^{i-1} \subset \cE, \cup \cE^i = \cE, \cE^i=0 \text{ for } i\gg 0 \rangle,$$
i.e. giving a vector bundle, together with a weighted filtration is
equivalent to giving a morphism $[\bA^1/\bG_m] \to \Bun_n^d$.
\end{lemma}
\begin{proof}
    We give the reformulation to fix the signs: A vector bundle $\cE$ on $C$ together with a weighted filtration $ \cE^i \subset \cE^{i-1} \subset \dots \subset \cE$ the Rees construction $\Rees(\cE^\bullet):= \oplus_{i\in \bZ} \cE^i x^{-i}$ defines an $\cO_C[x]$ module, i.e. a family on $C\times \bA^1$ which is $\bG_m$ equivariant for the action defined on the coordinate parameter with $\Rees(\cE^\bullet)|_{C \times 0} \cong \gr(\cE^\bullet)$.

    For the converse we argue as in Lemma \ref{equivG}: Any morphism $f\colon[\bA^1/\bG_m] \to \Bun_n^d$ defines a morphism $\bG_m \to \Aut_{\Bun_n}(f(0))$, i.e., a grading on the bundle $f(0)$ such that the corresponding filtration lifts canonically to the family.
\hfill $\Box$
\end{proof}

\subsubsection{Computing the numerical invariants}
Given a very close degeneration $[\bA^1/\bG_m] \to \Bun_n^d$ we can
easily compute $\wt(\cL_{\det}|_{f(0)})$, as follows. We use the
notation of the preceding lemma and write $\cE_i:=\cE^i/(\cE^{i+1})$
so that $f(0)=\oplus \cE_i$. Note that $\bG_m$ acts with weight $-i$
on $\cE_i$. Denoting further
$\mu(\cE_i):=\frac{\deg{\cE_i}}{\rk{\cE_i}}$ we find
\begin{align*}
 \wt(\cL_{\det}|_{f(0)}) &= -\wt_{\bG_m}(\det H^*(C,\oplus \cHom(\cE_i,\cE_j))),
\end{align*}
because $\cL_{det}$ was defined to be the dual of the determinant of
cohomology. As $\bG_m$ acts with weight $(i-j)$ on
$\cHom(\cE_i,\cE_j)$, it acts with the same weight on the cohomology
groups, taking determinants we find the weight $(i-j) \chi(
\cHom(\cE_i,\cE_j))$ on $\det H^*(C,\cHom(\cE_i,\cE_j))$. Thus by
Riemann-Roch we find
\begin{align*}
 \wt(\cL_{\det}|_{f(0)})&= \sum_{i,j} (j-i) \chi( \cHom(\cE_i,\cE_j))\\
 & = \sum_{i,j} (j-i)\big(\rk(\cE_i)\deg(\cE_j) - \rk(\cE_j)\deg(\cE_i) + \rk(\cE_i)\rk(\cE_j) (1-g)\big)\\
 &= 2\sum_{i<j} (j-i) ( \rk(\cE_i)\deg(\cE_j) - \rk(\cE_j)\deg(\cE_i)).
  \end{align*}
  As it is more common to express the condition in terms of the subbundles $\cE^l$ instead of the subquotients $\cE_i$ let us replace the factors $(j-i)$ by a summation over $l$ with $i\leq l <j$:
  \begin{align*}
 \wt(\cL_{\det}|_{f(0)}) &= 2 \sum_l (\sum_{i\leq l} \rk(\cE_i))(\sum_{j>l} \deg(\cE_j)) - (\sum_{i\leq l} \deg(\cE_i))(\sum_{j>l} \rk(\cE_j))\\
 &=  2 \sum_l \rk(\cE^l)(n-\rk(\cE^l)) (\mu(\cE^l) - \mu(\cE/\cE^l)).
\end{align*}
This is $<0$ unless $\mu(\cE^i) > \mu(\cE/\cE^i)$ for some $i$.
Conversely, if $\mu(\cE^i) > \mu(\cE/\cE^i)$ for some $i$, then the
two step filtration $0\subset \cE^i \subset \cE$  defines a very
close degeneration of positive weight. Thus we find the classical
condition:
\begin{lemma}
    A vector bundle $\cE$ is $\cL_{\det}$-stable if and only if for all $\cE^\prime \subset \cE$ we have $$\mu(\cE^\prime)< \mu(\cE).$$
\end{lemma}

\subsection{G-bundles on curves}
Let us formulate the analog for $G$-bundles, where $G$ is a
semisimple group over $k$ and we assume $k=\kbar$ to be
algebraically closed. We denote by $\Bun_G$ the stack of $G$-bundles
on $C$.

\subsubsection{The line bundle}
For the stability condition we need a line bundle on $\Bun_G$. One
way to construct a positive line bundle is to choose the adjoint
representation $\Ad:G\to \GL(\Lie(G))$, which defines for any
$G$-bundle $\cP$ its adjoint bundle $\Ad(\cP):=\cP\times^G \Lie(G)$
and set:  $$\cL_{\det}|_{\cP}:= \det H^*(C, \Ad(\cP))^\vee.$$ If $G$
is simple and simply connected, it is known that $\Pic(\Bun_G)\cong
\bZ$ (e.g. \cite{BiswasHoffmann}). In general $\cL_{\det}$ will not
generate the Picard group, but since our stability condition does
not change if we replace $\cL$ by a multiple of the bundle this line
bundle will suffice for us.

\subsubsection{Very close degenerations of $G$-bundles}
Recall from section \ref{Sec:GIT-quot} that for a cocharacter
$\lambda \colon \bG_m \to G$ we  denote by
$P_\lambda,U_\lambda,L_\lambda$ the corresponding parabolic
subgroup, its unipotent radical and the Levi subgroup.

To understand very close degenerations of bundles will amount to the
observation that Lemma \ref*{equivG} has an extension that holds for
families of bundles.

The source of degenerations is the following analog of the Rees
construction. Given $\lambda\colon \bG_m \to G$ we obtain a
homomorphism of group schemes over $\bG_m$:
\begin{align*}
\conj_\lambda\colon P_\lambda \times \bG_m &\to P_\lambda \times \bG_m\\
(p,t) &\mapsto (\lambda(t)p\lambda(t)^{-1},t).
\end{align*}
By \cite[Proposition 4.2]{Hesselink} this homomorphism extends to a
morphism of group schemes over $\bA^1$:
\begin{align*}
\gr_\lambda\colon P_\lambda \times \bA^1 &\to P_\lambda \times \bA^1
\end{align*}
in such a way that $\gr_{\lambda}(p,0)= \lim_{t\to 0}
\lambda(t)p\lambda(t)^{-1} \in L_\lambda \times 0$.

Moreover, these morphisms are $\bG_m$ equivariant with respect to
the action $(\conj_\lambda,\act)$ on $P_\lambda \times \bA^1$.

Given  a $P_\lambda$ bundle $\cE_\lambda$ on a scheme $X$ this
morphism defines a $P_\lambda$ bundle on $X \times [\bA^1/\bG_m]$
by:
$$\Rees(\cE_{\lambda},\lambda) := [\big((\cE_{\lambda} \times \bA^1)\times_{\bA^1}^{\gr_\lambda} (P_\lambda \times \bA^1)\big)/\bG_m],$$
where $\times_{\bA^1}^{\gr_\lambda}$ denotes the bundle induced via
the morphism $\gr_\lambda$, i.e., we take the product over $\bA^1$
and divide by the diagonal action of the group scheme
$P_\lambda\times \bA^1/\bA^1$, which acts on the right factor via
$\gr_\lambda$. By construction this bundle satisfies
$\Rees(\cE_{\lambda},\lambda)|_{X\times 1} \cong \cE_\lambda$ and
$$\Rees(\cE_{\lambda},\lambda)|_{X\times 0} \cong
\cE_{\lambda}/U_\lambda \times^{L_\lambda} P_\lambda$$ is the analog
of the associated graded bundle.

\begin{remark}
{\rm
    If $\lambda^\prime\colon \bG_m \to P_\lambda$ is conjugate to $\lambda$ in $P_\lambda$, say by an element $u\in U_{\lambda}$  then $P:=P_\lambda=P_\lambda^\prime$ and
    $$\gr_{\lambda^\prime}(p,t) = \gr_{\lambda}(upu^{-1},t).$$
    Therefore we also have $$\Rees(\cE_\lambda,\lambda) \cong \Rees(\cE_\lambda,\lambda^\prime),$$
    which tells us that the Rees construction only depends on the reduction to $P$ and the homomorphism $$\overline{\lambda}\colon \bG_m \to Z(P/U) \subset P/U.$$
    In the case $G=\GL(V)$ this datum is the analog of a weighted filtration on $V$, whereas $\lambda\colon \bG_m \to P \subset \GL(V)$ would define a grading on $V$.}
\end{remark}

Given a $G$-bundle $\cE$, a cocharacter $\lambda$ and a reduction
$\cE_\lambda$ of $\cE$ to a $P_\lambda$ bundle the $G$-bundle
$\Rees(\cE_\lambda,\lambda)\times^{P_\lambda} G$ defines a morphism
$f\colon [\bA^1/\bG_m]\to \Bun_G$ with $f(1)=\cE$. We claim that all
very close degenerations arise in this way:

\begin{lemma}\label{Lem:ReesG}
    Let $G$ be a split reductive group over $k$. Given a very close degeneration $f\colon [\bA^1/\bG_m] \to \Bun_G$ corresponding to a family $\cE$ of $G$-bundles on $X\times [\bA^1/\bG_m]$ there exist:
    \begin{enumerate}
        \item[\rm (1)] a cocharacter $\lambda\colon \bG_m \to G$, canonical up to conjugation,
        \item[\rm (2)] a reduction $\cE_\lambda$ of the bundle $\cE$ to $P_\lambda$,
        \item[\rm (3)] an isomorphism $\cE_\lambda \cong \Rees(\cE_{\lambda}|_{X\times 1},\lambda)$.
    \end{enumerate}
\end{lemma}
\begin{proof}
    Given $\cE$ we will again denote by $\cE_0:=\cE|_{X\times 0}$ and $\cE_1:=\cE|_{X\times 1}$.
    We define the group scheme $\cG^{\cE_0}:=\Aut_G(\cP_0/X)= \cE_0 \times^{G,\conj} G$. This is a group scheme over $X$ that is an inner form of $G\times X$. And the morphism $f|_{[0/\bG_m]}$ induces a morphism $\lambda_0\colon \bG_m\times X \to \Aut_G(\cE_0/X)= \cG^{\cE_0}$.

    As in Lemma \ref{equivG} (2) it is convenient to replace $\cE$ by the $\cG^{\cE_0}$-torsor $\cE^\prime:=\Isom_G(\cE,\cE_0)$. We know that $\lambda_0$ defines a parabolic subgroup $\cP_{\lambda_0} \subset \cG^{\cE_0}$ and the canonical reduction of $\cE^\prime_0$ to $\cP_{\lambda_0}$ lifts uniquely to a reduction $\cE^\prime_{\lambda}$ of $\cE^\prime$ by the same argument used in Lemma \ref{equivG}. The last step of the proof is then to consder the twisted action $\star\colon \bG_m \times \cE^\prime_\lambda \to \cE^\prime_\lambda$. Note that the fixed points for the action are simply the points in the Levi subgroup  $\cL_{\lambda_0} \subset  \cE^\prime_0=\cG^{\cE_0}$.

    The needed analog of the Bia\l ynicki-Birula decomposition is a result of Hesselinck \cite{Hesselink}: By the lemma we already know that for all geometric points $x$ of $X$ and $p\in \cE^\prime_\lambda|_x$ all limit points $\lim_{t\to 0}\lambda(t) \star p$ exist. On the other hand by \cite[Proposition 4.2]{Hesselink} the functor whose $S$-points are morphisms $S\times \bA^1 \to \cE^\prime_\lambda$ such that the restriction to $S\times \bG_m$ is given by the action of $\bG_m$ on $\cE^\prime_\lambda$ is represented by a closed subscheme of $\cE^\prime_\lambda$. Thus the functor is represented by $\cE^\prime_\lambda$.

    Thus the twisted action $\star$ on $\cE_\lambda^\prime$ extends to a morphism $\overline{\star}\colon \bA^1 \times \cE_\lambda^\prime \to \cE_\lambda^\prime$. In particular this induces $\bA^1 \times \cE_{\lambda,1} \to \cE_\lambda$ and thus a morphism of $\cP_{\lambda_0}$-bundles $\Rees(\cE_{\lambda,1},\lambda) \to \cE_\lambda$.

    This proves that the statement of the Lemma holds if we replace $G$ by $\cG^{\cE_0}$.

    To compare this with the description given in the lemma note that for any geometric point $x\in X(\kbar)$ the choice of a trivialization $\cE_{0,x} \cong G$ defines an isomorphism $G^{\cP_0} \cong G$ and therefore $\lambda_0|_x$ defines a defines a conjugacy class of cocharacters $\lambda\colon \bG_m \to G$. This conjugacy class is locally constant (and therefore does not depend on the choice of $x$) because we know from \cite[Expos\'e XI, Corollary 5.2bis]{SGA3} that the scheme parametrizing conjugation of cocharacters  $\Transp_G(\lambda_0,\lambda)$ is smooth over $X$. (This is the analog of the statement for vector bundles, that a $\bG_m$ action on $\cE$ allows to decompose $\cE=\oplus \cE_i$ as bundles, i.e. the dimension of the weight spaces of the fibers is constant over $x$.)
    This defines $\lambda$. To conclude we only need to recall that reductions of $\cE^\prime$ to $\cP_{\lambda_0}$ correspond to reductions of $\cE$ to $P_\lambda$:
\begin{lemma}
\label{lemma1.14}
   Let $\cG \to X$ be a reductive group scheme, $\lambda \colon \bG_{m,X}\to \cG$ a cocharacter and $\cE$ a $\cG$-torsor over $X$. Then a natural bijection between:
   \begin{enumerate}
    \item[\rm (1)] Reductions $\cE_\lambda$ of $\cE$ to $P_\lambda$,
    \item[\rm (2)] Parabolic subgroups $\cP \subset \cG^{\cE}=\Aut_{\cG}(\cE/X)$ that are locally conjugate to $P_\lambda$,
    \end{enumerate}
    is given by $\cE_\lambda \mapsto \cP:=\Aut_{P_\lambda}(\cE_\lambda/X) \subset \Aut_{\cG}(\cE/X)$.
\end{lemma}

\begin{proof}[Proof of Lemma \ref{lemma1.14}]
    A reduction of $\cE$ is a section $s \colon X \to \cE/P_{\lambda}$. Note that $\cG^{\cE}$ acts on $\cE/P_{\lambda}$ and $\Stab_{\cG^{\cE}}(s) \subset \cG^{\cE}$ is a parabolic subgroup that is locally of the same type as $P_{\lambda}$, becasuse this holds if $\cE$ is trivial and $s$ lifts to a section of $\cE$. Locally in the smooth topology we may assume these conditions.

    Similarly given $\cP \subset \cG^{\cE}$ locally the action of $\cP$ on $\cE/P_\lambda$ has a unique fixed point and this defines a section.
\hfill $\Box$
\end{proof}
    Using the lemma we find that both sections of $\cE^\prime/\cP_{\lambda_0}$ and sections of $\cE/P_\lambda$ correspond to parabolic subgroups of $\Aut_G(\cE)=\Aut_{\cG^{\cE_0}}(\cE^\prime)$. This proves the lemma.
\hfill $\Box$
\end{proof}
\begin{remark}
{\rm    Note that in the above result we assumed that $G$ is a split group. In general, we saw that the natural subgroup that contains a cocharacter is $\Aut_G(\cE_0)$, which is an inner form of $G$ over $X$. In particular it may well happen that $G$ does not admit any cocharacter or any parabolic subgroup.

    This is apparent for example in the case $k=\bR$ and $G=U(n)$. For a $G$-torsor $\cE$ on $C$ we will find a canonical reduction to a parabolic subgroup $P_{\kbar}\subset G_{\kbar}$ but the descent datum will then only be given for $\cP \subset \Aut_G(\cE/X)$.}
\end{remark}

\subsubsection{The numerical criterion}
Finally we have to compute the weight of $\cL_{\det}$ on very close
degenerations.

The computation is the same as for vector bundles and for the
criterion it is sometimes convenient to reduce it to reductions for
maximal parabolic subgroups. Let us  choose $T\subset B \subset G$ a
maximal torus and a Borel subgroup and $\lambda \colon \bG_m \to G$
a dominant cocharacter, i.e. $\langle\lambda,\alpha\rangle \geq 0$
for all roots such that $\cg_\alpha \in \Lie(B)$. Let us denote by
$I$ the set of positive simple roots with respect to $(T,B)$ and by
$I_P:=\{\alpha_i \in I | \lambda(\alpha_i)=0\}$ the simple roots
$\alpha_i$ for which $-\alpha_i$ is also a root of $P_\lambda$. For
$j\in I$ let us denote by $\check{\omega}_j\in X_*(T)_\bR$ the
cocharacter defined by $\check{\omega}_j(\alpha_i) = \delta_{ij}$.
And by $P_j$ the corresponding maximal parabolic subgroup.

Then $\lambda \colon \bG_m \to Z(L_\lambda) \subset L_\lambda
\subset P_\lambda$. Thus for any very close degeneration $f\colon
[\bA^1/\bG_m] \to \Bun_G$ given by $\Rees(\cE_\lambda,\lambda)$ the
bundle $\cL_{\det}$ defines a morphism
$$ \wt_{\cL} \colon X_*(Z_\lambda) \subset \Aut_{\Bun_G}(f(0)) \to \bZ.$$

Then $\lambda= \sum_{j\in I-I_P} a_j \check{\omega}_j$ for some $a_j
>0$.

\begin{align*}
\wt(\cL_{\det}|_{f(0)}) &= \wt_{\cL}(\lambda) = \sum_{j\in I-I_P}
a_j \wt_{\cL}(\check{\omega}_j).
\end{align*}
For each $j$ we get a decomposition $\Lie(G) = \oplus_i \Lie(G)_i$,
where $\Lie(G)_i$ is the subspace of the Lie algebra on which
$\check{\omega}_j$ acts with weight $i$. Each of these spaces is a
representation of $L_\lambda$ and also of the Levi subgroups $L_j$
of $P_j$. Using this decomposition we find as in the case of vector
bundles:



\begin{align*}
\wt_{\cL}(\check{\omega}_j) &= -\wt_{\bG_m}(\det H^*(C,\oplus \cE_{0,\lambda} \times^{L_\lambda} \Lie(G)_i)) \\
&= \sum_{i} i \cdot \chi( \cE_{0,\lambda} \times^{L_\lambda} \Lie(G)_i)\\
& =\sum_{i} i  (\deg(\cE_{0,\lambda} \times^{L_\lambda} \Lie(G)_i) + \dim(\Lie(G)_i)(1-g))\\
&= 2\sum_{i>0} i (\deg(\cE_{0,\lambda} \times^{L_\lambda} \Lie(G)_i)\\
\end{align*}
Now $\deg(\cE_{0,\lambda} \times^{L_\lambda} \Lie(G)_i) =
(\deg(\det(\cE_{0,\lambda} \times^{L_\lambda} \Lie(G)_i))$.  Since
the Levi subgroups of maximal parabolics have only a one dimensional
space of characters, all of these degrees are positive multiples of
$\det(\Lie(P_j))$. Thus we find the classical stability criterion:
\begin{corollary}
    A $G$-bundle $\cE$ is $\cL_{\det}$-stable if and only if for all reductions $\cE_P$ to maximal parabolic subgroups $P\subset G$ we have $\deg(\cE_P \times^P \Lie(P)) <0$.
\end{corollary}


\subsubsection{Parabolic structures}

Parabolic $G$-bundles are $G$-bundles equipped with a reduction of
structure group at a finite set of closed points.

Let us fix notation for these. We keep our reductive group $G$, the
curve $C$ and a finite set of rational points $\{x_1,\dots,
x_n\}\subset C(k)$ and parabolic subgroups $P_1,\dots,P_n \subset
G$.

$$\Bun_{G,\un{P},\un{x}}(S):=\left\langle (\cE,s_1,\dots,s_n) | \cE \in \Bun_G(S), s_i\colon S \to \cE_{x_i\times S}/P_i \text{ sections }\right\rangle$$

The forgetful map $\Bun_{G,\un{P},\un{x}} \to \Bun_G$ is a smooth
proper morphism with fibers isomorphic to $\prod_i G/P_i$.

In particular, very close degenerations of a parabolic $G$ bundle
are uniquely defined by a very close degenerations of the underlying
$G$-bundle.

There are more line bundles on $\Bun_{G,\un{P},\un{x}}$, namely any
dominant character of $\chi_i\colon P_i \to \bG_m$ defines a
positive line bundle on $G/P_i$ and this induces a line bundle on
$\Bun_{G,\un{P},\un{x}}$.

The weight of a this line bundle on a very close degeneration, is
given by the pairing of $\chi_i$ with the one parameter subgroup in
$\Aut_{P_i}(\cE_{x_i})$.

We will come back to this in the section on parahoric bundles
(Section \ref{sec:stabcond}).
\subsection{The example of chains of bundles on curves}
We briefly include the example of chains of bundles as an easy
example of a stability condition that depends on a parameter.

Again we fix a curve $C$. A holomorphic chain of length $r$ and rank
$\un{n} \in \bN^{r+1}$ is the datum $(\cE_i,\phi_i)$ where
$\cE_0,\dots,\cE_r$ are vector bundles of rank $n_i$ and
$\phi_i\colon \cE_i \to \cE_{i-1}$ are morphisms of $\cO_C$-modules.
The stack of chains is denoted $\Chain_{\un{n}}$. It is an algebraic
stack, locally of finite type. One way to see this is to show that
the forgetful map $\Chain_{\un{n}} \to \prod_{i=0}^r \Bun_{n_i}$ is
representable. As for the stack $\Bun_n^d$  all chains admit scalar
automorphisms $\bG_m$, so we will need to look for line bundles on
which these automorphisms act trivially.

The forgetful map to $\prod_{i=0}^r \Bun_{n_i}$ already gives a many
line bundles on $\Chain_{\un{n}}$, as we can take products of the
pull backs of the line bundles $\cL_{\det}$ on the stacks
$\Bun_{n_i}$. Somewhat surprisingly these are only used in
\cite{Schmitt}, whereas the standard stability conditions (e.g.,
\cite{AGPS}) arise  from the following bundles:
\begin{enumerate}
    \item[\rm (1)] $\cL_{\det}:= \det( H^*(C, \End(\oplus \cE_i)))^\vee$
    \item[\rm (2)] Fix any point $x\in C$ and $i=1,\dots r$. Set $\cL_i\coloneq \det(\Hom(\oplus_{j\geq i}\cE_{j,x},\oplus_{l<i}\cE_{l,x}))$.
\end{enumerate}

\begin{remark} {\rm Note that on all of these bundles the central automorphism group $\bG_m$ of a chain acts trivially and one can check that up to the multiple $[k(x):k]$ the Chern classes of the bundles $\cL_i$ do not depend on $x$. We will not use this fact.

The choice of the bundles $\cL_i \in \Pic(\Chain_{\un{n}})$ is made
to simplify our computations. From a more conceptual point of view
the lines $\cL_{n_i}:=\det(\cE_{i,x})$ define bundles on
$\Bun_{n_i}$ which are of weight $n_i$ with respect to the central
automorphism group $\bG_m$. The pull backs of these bundles generate
a subgroup of $\Pic(\Chain_{\un{n}})$ and the $\cL_i$ are a basis
for the bundles of weight $0$ in this subgroup.}
\end{remark}

To classify maps $[\bA^1/\bG_m] \to \Chain_{\un{n}}$ note that
composing with $\forget\colon \Chain_{\un{n}} \to \prod \Bun_{n_i}$
such a morphism induces morphisms $[\bA^1/\bG_m] \to \Bun_{n_i}$,
which we already know to correspond to weighted filtrations of the
bundles $\cE_i$ and a lifting of a morphism  $[\bA^1/\bG_m] \to
\prod \Bun_{n_i}$ to $\Chain_{\un{n}}$ is given by homomorphisms
$\phi_i \colon \cE_i \to \cE_{i-1}$ that respect the filtration.

Thus we find that a very close degeneration of a chain $\cE_\bullet$
is a weighted filtration $\cE_\bullet^{i} \subset \cE_\bullet$ and
we already computed
\begin{align*}
\wt(\cL_{\det}|_{\gr(\cE_\bullet)}) &= 2 \sum_{i}  \rk(\cE^i_\bullet)(n-\rk(\cE^i_\bullet)) (\mu(\cE_\bullet/\cE^i_\bullet) - \mu(\cE^i_\bullet)) \\
&= 2 \sum_{i}  \rk(\cE^i_\bullet)n (\mu(\cE_\bullet) -
\mu(\cE^i_\bullet)).
\end{align*}
Further we have:
\begin{align*}
\wt(\cL_{j}|_{\gr(\cE_\bullet)}) &= \sum_{i} (\sum_{l\geq j}
\rk(\cE^i_l)n) - n^i(\sum_{l\geq j} \rk(\cE^i)).
\end{align*}

Thus we find that a chain $\cE_\bullet$ is $\cL_{\det} \tensor
\cL_i^{2m_i}$-stable if and only if for all subchains
$\cE^\prime_\bullet \subset \cE_\bullet$ we have
$$ \frac{\sum_i \deg(\cE_i^\prime) + \sum_j m_j \rk(\oplus_{l\geq j} \cE_l^\prime)}{\sum_i \rk(\cE_i^\prime)} < \frac{\sum_i \deg(\cE_i) + \sum_j m_j \rk(\oplus_{l\geq j} \cE_l)}{\sum_i \rk(\cE_i)}.$$
This is equivalent to the notion of $\alpha-$stability used in
\cite[Section 2.1]{AGPS}.

\begin{remark}
{\rm    Also for the moduli problem of coherent systems on $C$, i.e. pairs $(\cE,V)$ where $\cE$ is a vector bundle of rank $n$ on $C$ and $V \subset H^0(C,\cE)$ is a subspace of dimension $r$ one recovers the stability condition quite easily:
    Families over $S$ are pairs $(\cE,\cV, \phi\colon \cV \tensor \cO_C \to \cE)$ where $\cE$ is a family of vector bundles on $X\times S$, $\cV$ is a vector bundle on $S$ and we drop the condition that $\phi$ corresponds to an injective map $\cV\to \pr_{S,*} \cE$. We denote this stack by $\text{CohSys}_{n,r}$.
    There are natural forgetful maps $\text{CohSys}_{n,r} \to \Bun_n$ and $\text{CohSys}_{n,r}\to B\GL_n$ induced respectively by the bundles $\cE$ and $\cV$. As above for any point $x\in X$ we obtain a line bundle $\det(\cV)^n \tensor \det(\cE_x)^{-r}$ on $\text{CohSys}_{n,r}$ and together with $\cL_{\det}$ one then recovers the classical stabilty condition that one finds for example in \cite[Definition 2.2]{Bradlow}.}
\end{remark}
\subsection{Further examples}
Other, more advanced examples can be found in the article
\cite[Section 4.2]{DHL}. For example this contains an argument, how
the Futaki invariant introduced by Donaldson arises from the point
of view of algebraic stacks.

\section{A criterion for separatedness of the stable locus}
We now want to give a criterion which guarantees that the set of
$\cL$-stable points is a separated substack if the stack $\cM$ and
the line bundle $\cL$ satisfy suitable local conditions (Proposition
\ref{Prop:sep}). The article \cite{MT} by Martens and Thaddeus was
an important help to find the criterion. Again, the proofs turn out
to be quite close to arguments that already appear in Mumford's
book.
\subsection{Motivation from the valuative criterion}
Let us first sketch the basic idea. Let $\cM$ be an algebraic stack.
For the valuative criterion for separatedness one considers pairs
$f,g\colon \bD = \Spec k\pbl t\pbr  \to \cM$ together with an
isomorphism $f|_{\Spec k(\!(t)\!)} \cong g|_{\Spec k(\!(t)\!)}$ and
tries to prove that for such pairs $f\cong g$. The basic datum is
therefore a morphism
$$f\cup g \colon \bD\cup_{\mathring{\bD}} \bD \to \cM.$$
Note that the scheme $ \bD\cup_{\mathring{\bD}} \bD$ is a completed
neighborhood of the origins in the affine line with doubled origin
$$\bA^1 \cup_{\bG_m} \bA^1 \cong [\bA^2-\{0\}/\bG_m,(t,t^{-1})]
\subset [\bA^2/\bG_m].$$

Thus also the union of two copies of $\bD$ along their generic point
is naturally an open subscheme of a larger stack:
$$\bD\cup_{\mathring{\bD}} \bD \subset [(\bA^2 \times_{\bA^1}\bD)/\bG_m] = [\Spec (k[x,y]\tensor_{k[\pi],\pi=xy} k\pbl \pi\pbr ) /\bG_m,(t,t^{-1})],$$ where the right hand side inserts a single point $[0/\bG_m]$.

Further, the coordinate axes $\Spec k[x],\Spec k[y] \subset \Spec
k[x,y]=\bA^2$ define closed embeddings $[\bA^1/\bG_m] \to [(\bA^2
\times_{\bA^1}\bD)/\bG_m]$ that intersect in the origin $[0/\bG_m]$.
As the weights of the $\bG_m$ action on the two axes are $1$ and
$-1$ we again see that for any line bundle $\cL$ on this stack the
weights of the restriction to the two copies of $[\bA^1/\bG_m]$
differ by a sign.

In terms of $\cL$-stability on $\cM$ this means that whenever a
morphism $f,g\colon \bD\cup_{\mathring{\bD}} \bD \to \cM$ extends to
$ [(\bA^2 \times_{\bA^1}\bD)/\bG_m]$, only one of the two origins
can map to an $\cL$-stable point.

As the complement of the origin $[0/\bG_m]$ is of codimension $2$ in
the above stack one could  expect  that a morphism $f\cup g\colon
\bD\cup_{\mathring{\bD}} \bD  \to \cM$ extends to a morphism of some
blow up of  $[(\bA^2 \times_{\bA^1}\bD)/\bG_m]$ centered at the
origin. This will be the assumption that we will impose on $\cM$.

The basic observation then is that the exceptional fiber of such a
blow up can be chosen to be a chain of equivariant projective lines
and it will turn out that the weight argument indicated above still
works for such chains if the line bundle $\cL$ satisfies a numerical
positivity condition.

\subsection{The test space for separatedness and equivariant blow ups}
Let $R$ be a discrete valuation ring together with a local parameter
$\pi \in R$, $K:=R[\pi^{-1}]$ the fraction field and $k=R/(\pi)$ the
residue field.

As for the affine line, the scheme $\Spec R$ has a version with a
doubled special point $\ST_R := \Spec R \cup_{\Spec K} \Spec R$, the
test scheme for separateness. The analog of $[\bA^2/\bG_m]$ is given
as follows. The multiplicative group $\bG_m$ acts on
$R[x,y]/(xy-\pi)$ via $t.x:=tx, t.y:=t^{-1}y$. Let us denote
$$\overline{\ST}_R := [\Spec (R[x,y]/(xy-\pi)) /\bG_m].$$
As before we have:
\begin{enumerate}
    \item[\rm (1)] two open embeddings
    $$\xymatrix{
    j_x \colon \Spec R \ar@{^(->}[r]^{\circ}\ar[d]^{\cong} &  \overline{\ST}_R \\
    [\Spec R[x,x^{-1}]/\bG_m] \ar[r]^-{\cong} & [\Spec R[x,x^{-1},y]/(xy-\pi) /\bG_m]\ar@{^(->}[u]^{\circ} }$$
    $$\xymatrix{
        j_y \colon \Spec R \ar@{^(->}[r]\ar[d]^{\cong} &  \overline{\ST}_R \\
        [\Spec R[y,y^{-1}]/\bG_m] \ar[r]^-{\cong} & [\Spec R[x,y,y^{-1}]/(xy-\pi) /\bG_m]\ar@{^(->}[u]^{\circ} }$$
    that coincide on $\Spec K$.
    \item[\rm (2)] two closed embeddings:
    $$i_x\colon [\bA^1_k/\bG_m] \cong [\Spec k[x]/\bG_m] \cong [\Spec R[x,y]/(y,xy-\pi) /\bG_m] \subset \overline{\ST}_R$$
    $$i_y\colon [\bA^1_k/\bG_m] \cong [\Spec k[y]/\bG_m] \cong [\Spec R[x,y]/(x,xy-\pi) /\bG_m] \subset \overline{\ST}_R$$
    and the intersection of these is $[\Spec k/\bG_m]=:[0/\bG_m]$.
\end{enumerate}
We will need to understand blow ups of $ \overline{\ST}_R$ supported
in $[\Spec k/\bG_m]$. For this, let us introduce some notation. A
{\em chain of projective lines} is a scheme $E= E_1 \cup \dots \cup
E_n$ where $E_i\subset E$ are closed subschemes, together with
isomorphisms $\phi_i\colon E_i\map{\cong} \bP^1$ such that $E_i\cap
E_{i+1} = \{x_i\}$ is a reduced point with $\phi_i(x_i) = \infty,
\phi_{i+1}(x_i)=0$. An equivariant chain of projective lines is a
chain of projective lines together with an action of $\bG_m$ such
that for each $i$ the action induces the standard action of some
weight $w_i$ on $\bP^1 =\Proj k[x,y]$, i.e. this is given by
$t.x=t^{w_i+d}x$, $t.y=t^{d}y$ for some $d$.

We say that an equivariant chain is of negative weight, if all the
$w_i$ are negative. In this case for all $i$ the points
$\phi_i^{-1}(0)$ are the repellent fixed points in $E_i$ for $t\to
0$ .

\begin{lemma}\label{Lem:EquivBl}
Let $I \subset R[x,y]/(xy-\pi)$ be a $\bG_m$ invariant ideal
supported in the origin $\Spec k = \Spec R[x,y]/(x,y,\pi)$. Then there exists
an invariant ideal $\tilde{I}$ and a blow up
$$ p \colon \Bl_{\tilde{I}}(\oST_R)= [\Bl_{\tilde{I}}\big(\Spec R[x,y]/(xy-\pi)\big)/\bG_m] \to \oST_R$$
dominating $\Bl_{I}$ such that $p^{-1}([0/\bG_m]) \cong [E/\bG_m]$
where $E$ is a chain of projective lines of negative weight.
\end{lemma}
\begin{proof}
As $I$ is supported in $(x,y)$ there exists $n$ such that $(x,y)^n
\subset I$ and as it is $\bG_m$ invariant, it is homogeneous with
respect to the grading for which $x$ has weight $1$ and $y$ has
weight $-1$. Let  $P(x,y)= \sum_{i=l}^N a_i x^{d+i}y^i $ be a
homogeneous generator of weight $d\geq 0$ with $a_l\neq 0$ that is
not a monomial. We may assume $a_i\in R^*$ as $\pi=xy$. We claim
that then $x^{d+l}y^l \in I$. As $(x,y)^n \subset I$ we may assume
that $d+2N<n$. But then $P(x,y)-(a_{l+1}/a_l) xy P(x,y)\in I$ is an
element for which the coefficient of $x^{d+l+1}y^{l+1}$ vanishes.
Inductively this shows that $x^{d+l}y^l \in I$, so that $I$ is
monomial.

Write $I=(x^n,y^m,x^{n_i}y^{m_i})_{i=1,\dots N}$ with $n_i<n,m_i<m$.
This ideal becomes principal after sucessively blowing up $0$ and
then blowing up $0$ or $\infty$ in the exceptional $\bP^1$'s:
Blowing up $(x,y)$ we get charts with coordinates $(x,y) \mapsto
(x^\prime y,y)$ and $(x,y) \mapsto (x,xy^\prime)$. Since $x$ has
weight $1$ and $y$ has weight $-1$ we see that the weights of
$(x',y)$ are $(2,-1)$ and the weights of $(x,y')$ are $(1,-2)$.

In the first chart the proper transform of $I$ is $(x^{\prime
n}y^n,y^m,x^{\prime n_i}y^{m_i+n_i})_{i=1,\dots N}$. This ideal is
principal if $m=1$ and otherwise equal to an ideal of the form
$$y^k(y^{m-k}, \text{ mixed monomials of lower total degree}).$$ A
similar computation works in the other chart. By induction this
shows that the ideal will become principal after finitely many blow
ups and that in each chart the coordinates $(x^{(i)},y^{(i)})$ have
weights $(w_i,v_i)$ with $w_i>v_i$.
\hfill $\Box$
\end{proof}

\begin{remark}{\rm
    Let $\cL$ be a line bundle on $[\bP^1/(\bG_m,\act_w)]$ then $$\deg(\cL|_{\bP^1}) = \frac{1}{w}\big(\wt_{\bG_m}(\cL|_{\infty}) - \wt_{\bG_m}(\cL|_{0})\big).$$}
\end{remark}
\begin{proof}
    Write $d=\deg(\cL|_{\bP^1})$.
    Let $\bA^1_0=\Spec k[x], \bA^1_{\infty}=\Spec k[y]$ be the two coordinate charts of $\bP^1$ and $e_x \in \cL(\bA^1_0), e_y\in \cL(\bA^1_\infty)$ two generating sections such that $e_y=x^{d}e_x$ on $\Spec k[x,x^{-1}]$.

    Then $\act^{\#}(e_y)=t^{\wt_{\bG_m}(\cL|_{\infty})}e_y$ and $\act^{\#}(x^de_x)= t^{wd} t^{\wt_{\bG_m}(\cL|_{0})}x^de_x$. Thus we find
    $$ t^{wd}= t^{\wt_{\bG_m}(\cL|_{\infty})-\wt_{\bG_m}(\cL|_{0})}.$$
    \hfill $\Box$
\end{proof}

Combining the above computations we propose the following
definitions:

\begin{definition}\label{Def:almost_proper}{\rm
    Let $\cM$ be an algebraic stack, locally of finite type with affine diagonal. We say that $\cM$ is {\em almost proper} if
    \begin{enumerate}
    \item[\rm (1)] For all valuation rings $R$ with field of fractions $K$  and $f_K\colon \Spec K \to \cM$ there exists a finite extension $R^\prime/R$ and a morphism $f \colon \Spec R^\prime \to \cM$ such that $f|_{\Spec K^\prime } \cong f_K|_{\Spec K^\prime}$ and
    \item[\rm (2)] for all complete discrete valuation rings $R$ with algebraically closed residue field and all morphisms $f\colon ST_R \to \cM$ there exists a blow up $\Bl_{\tilde{I}}(\oST_R)$ supported at $0$ such that $f$ extends to a morphism $\overline{f} \colon \Bl_{\tilde{I}}(\oST_R) \to \cM$.

    \end{enumerate}
    Given a line bundle $\cL$ on an almost proper algebraic stack $\cM$ we say that $\cL$ is {\em nef on exceptional lines} if
    \begin{enumerate}
        \item[]for all $f\colon \ST_R \to \cM$ the extension $\overline{f}$ from (2) can be chosen such that for all equivariant projective lines $E_i$ in the exceptional fiber of the blow up we have $\deg(\cL|_{E_i})\geq 0.$
    \end{enumerate}}

\end{definition}

\begin{remark}{\rm
    Note that for schemes (and also algebraic spaces) of finite type the above definition reduces to the usual valuative criterion for properness: We already saw in Lemma \ref{Lem:vcdeformation} that in case $X=\cM$ is a scheme, any morphism from an equivariant projective line to $X$ must be constant and more generally if $\bG_m$ acts on $\Spec R[x,y]/(xy-\pi))$ with positive weight on $x$ and negative weight on $y$ the morphism $[(\Spec R[x,y]/(xy-\pi))/\bG_m] \to \Spec R$ is a good coarse moduli space (\cite{Alper}), so for any algebraic space $X$ any morphism $\Bl_{\tilde{I}}(\oST_R) \to X$ factors through $\Spec(R)$. Therefore the existence of a morphism as in (2) implies the valuative criterion for properness (\cite[Tag 0ARL]{Stacks}).

    Similarly (2) above could be used to define a notion of almost separatedness for stacks.}
\end{remark}
\begin{remark}{\rm 
    In the above definition we could have replaced condition (2) by the condition:
    \begin{enumerate}
        \item[\rm (2')] For all discrete valuation rings $R$ and all morphisms $f\colon ST_R \to \cM$ there exists a finite extension $R^\prime/R$ and a blow up $\Bl_{\tilde{I}}(\oST_{R^\prime})$ supported at $0$ such that $f$ extends to a morphism $\overline{f} \colon \Bl_{\tilde{I}}(\oST_{R^\prime}) \to \cM$.
    \end{enumerate}
    These two conditions are equivalent for stacks locally of finite type, we put (2) because it is sometimes slightly more convenient to check.}
\end{remark}
\begin{proof}
    We only need to show that (2) implies (2'). The standard argument for schemes can be adapted here with some extra care taking into account automorphisms of objects:
    Note first that because $\cM$ is locally of finite type it suffices to prove (2') in the case where the DVR $R$ has an algebraically closed residue field. In this case we denote by $\hat{R}$ the completion of $R$ and by $f_{\hat{R}}$ the restriction of $f$ to $\ST_{\hat{R}}$. This map has an extension $\overline{f}_{\hat{R}}\colon Bl_{{I}}(\oST_{\hat{R}}) \to \cM$ and we have seen in Lemma \ref{Lem:EquivBl} that we may assume that $I$ is already defined over $R$, i.e., that $ Bl_{{I}}(\oST_{\hat{R}})=  Bl_{{I}}(\oST_{R})_{\hat{R}}$.

    As the extension $\hat{R}/R$ is faithfully flat, to define an extension $\overline{f}\colon Bl_{{I}}(\oST_{R}) \to \cM$ of $f$ we need to define a descent datum for $\overline{\hat{f}}$, i.e. for the two projections $p_{1,2}\colon Bl_{{I}}(\oST_{R})_{\hat{R} \tensor_R \hat{R}} \to Bl_{{I}}(\oST_{R})_{\hat{R}}$ we need an element $\phi \in \Isom_{\cM}(p_1^*(\overline{\hat{f}}),p_2^*(\overline{\hat{f}}))$ that over $ST_{{\hat{R} \tensor_R \hat{R}} }$ coincides with the one given by $f$.

    Now note that if we denote by $\hat{K}$ the fraction field of $\hat{R}$ there is a cartesian diagram of rings (see e.g., \cite[Lemma 5]{Unif})
    $$\xymatrix{
        {\hat{R} \tensor_R \hat{R}}\ar[r]\ar[d]^{\text{mult}} & \hat{K} \tensor_R \hat{K}\ar[d]^{\text{mult}} \\
        \hat{R} \ar[r] & \hat{K}.
        }$$
    Also by our assumptions on the diagonal of $\cM$ we know that $\Isom(p_1^*\overline{\hat{f}},p_2^*\overline{\hat{f}} ) \to  Bl_{{I}}(\oST_{R})_{\hat{R} \tensor_R \hat{R}}$ is affine and we have canonical sections of this morphism over the diagonal $\Delta\colon Bl_{{I}}(\oST_{\hat{R}})  \to  Bl_{{I}}(\oST_{R})_{\hat{R} \tensor_R \hat{R}}$ and over $Bl_{{I}}(\oST_{R})_{\hat{K} \tensor_R \hat{K}} = \ST_{\hat{K} \tensor_R \hat{K}}$ we have the section defined by the descent datum for $f$. As these agree on the intersections $\Delta: \Spec \hat{K} \to \Spec \hat{K} \tensor_R \hat{K}$ the cartesian diagram implies that these define $\phi$. This map is a cocycle, because this holds over the open subscheme $\ST_R$.
    \hfill $\Box$
\end{proof}

Recall from Remark \ref{Rem:L-stable} that $\cL$-stable points
define an abstract substack $\cM^s\subset \cM$, for these points
Definition \ref{Def:almost_proper} implies the valuative criterion:
\begin{proposition}[Separatedness of stable points]\label{Prop:sep}
    Let $\cM$ be an algebraic stack locally of finite type over $k$ with affine diagonal and $\cL$ a line bundle on $\cM$. Suppose that $\cM$ is almost proper and $\cL$ is nef on exceptional lines.
Then the stack of stable points $\cM^s\subseteq \cM$ satisfies the
valuative criterion for separatedness, i.e., for any complete
discrete valuation ring $R$ with fraction field $K$  and
algebraically closed residue field any morphism $\ST_R \to \cM^s$
factors through $\Spec R$.
\end{proposition}
\begin{proof}
    We need to show that for any morphism $f \colon \ST_R \to \cM$ such that all points in the image of $f$ are stable we have $j_x\circ f \cong j_y\circ f$.

    Let us first show that the conditions imply that the morphisms coincide on closed points, i.e., $f(j_x((\pi))) \cong f(j_y((\pi))) \in \cM$.

    If $f$ extends to a morphism $\overline{f}\colon \oST_R \to \cM$, we have $\wt(i_x^*f^*\cL)=-\wt(i_y^*f^*(\cL))$. As we assumed that both closed points are stable this shows that neither $i_x\circ f$ nor $i_y\circ f$ can be a very close degeneration, i.e. $f(j_x((\pi))) \cong \overline{f}(i_x(0)) = \overline{f}(i_y(0)) \cong f(j_y((\pi)))$.

    If $f$ does not extend, then by assumption there exists an extension $\overline{f}\colon \Bl_I(\oST_R) \to \cM$ such that $\deg(\cL|_{E_i})\geq 0$ on all equivariant $\bP^1$'s contained in the exceptional fiber of the blow up.

    Now for line bundles $\cL^\prime$ on $[\bP^1/\bG_m,\act_d]$ we saw that $\deg(\cL^\prime) = \frac{1}{d} \big(\wt_{\bG_m}(\cL|_\infty) - \wt_{\bG_m}(\cL|_{0})\big)$.
    Thus if we order the fixed points $x_0,\dots,x_{n}$ of the $\bG_m$-action on the chain $E_i$ such that $x_0$ is the point in the proper transform of the $x-$axis and $x_i,x_{i+1}$ correspond to $0,\infty$ in $E_i$ we find that $x_n$ corresponds to the proper transform of the $y$-axis. Note that $x_0$ being the repellent fixed point of $E_1$ we have $\wt(i_x^*\overline{f}^*\cL)=\wt_{\bG_m}(\overline{f}^*\cL|_{x_0})$  and similarly $\wt_{\bG_m}(\overline{f}^*\cL|_{x_k})=-\wt(i_y^*\overline{f}^*\cL)$.
    Finally the condition $\deg(\cL|_{E_i})\geq 0$ implies $\wt_{\bG_m}(\overline{f}^*\cL|_{x_i}) \geq \wt_{\bG_m}(\overline{f}^*\cL|_{x_{i+1}})$ for all $i$ as the exceptional divisors are of negative weight.
    Thus we find:
    $$\wt(i_x^*\overline{f}^*\cL)=\wt_{\bG_m}(\overline{f}^*\cL|_{x_0}) \geq \dots \geq \wt_{\bG_m}(\overline{f}^*\cL|_{x_n})=-\wt(i_y^*\overline{f}^*\cL).$$
    Now if $f(j_x((\pi)))\not\cong f(x_0)$ is stable, we know that $\wt_{\bG_m}(\overline{f}^*\cL|_{x_0}) <0$ so that $ \wt_{\bG_m}(\overline{f}^*\cL|_{x_n}) = - \wt(i_y^* f^*(\cL)) <0$, contradicting stability of $f(j_y((\pi)))$.

    Thus we find that $\deg(\cL|_{E_i})=0$ for all $i$ and $f(j_x((\pi)))\cong f(x_0)$. Then $\overline{f}(x_0)$ is stable so that $\overline{f}|_{E_1-\{x_1\}}$ must be constant and we inductively find that $f(x_0) \cong \dots \cong f(x_n)\cong  f(j_y((\pi)))$ and that $\overline{f}$ is constant on the exceptional divisor, so $f$ does extend to $\oST$.

    To conclude that this implies $j_x\circ f \cong j_y\circ f$ choose an affine scheme of finite type $p\colon \Spec(A) \to \cM$ such that $p$ is smooth and $f(j_x((\pi)))=\overline{f}(0)=f(j_y((\pi))) \in \im(p)$. Choose moreover a lift $x\in \Spec(A)$ of $\overline{f}([0/\bG_m])$. By smoothness we can inductively lift $\overline{f}|_{[(\Spec R[x,y]/(\pi^n,xy-\pi))/\bG_m]}$ to a morphism $\tilde{f}_n \colon [(\Spec R[x,y]/(\pi^n,xy-\pi))/\bG_m] \to \Spec A$. All of these maps factor through their coarse moduli space $[(\Spec R[x,y]/(\pi^n,xy-\pi))/\bG_m] \to \Spec R/(\pi^n)\to \Spec A$, defining a map $\Spec R \to \Spec A$ and thus $\tilde{f} \colon \oST_R \to \Spec R \to \Spec A$ that lifts both $j_x\circ f$ and $j_y\circ f$, so these maps coincide.
    \hfill $\Box$
\end{proof}
\begin{remark} {\rm 
    For semistable points that are not stable the above computation also suggests to define a notion of S-equivalence, as it shows that in an almost proper stack any two semistable degenerations could be joined by a chain of projective lines. If $\cL$ is nef on exceptional lines the line bundle would restrict to the trivial bundle on such a chain. In the examples this reproduces the usual notion of S-equivalence.}
\end{remark}

Before giving examples let us note that the Keel-Mori theorem now
implies the existence of coarse moduli spaces for $\cM^s$ in many
situations:

\subsection{An existence result for coarse moduli spaces}

\begin{proposition}\label{Prop:coarse}
    Let ($\cM,\cL$) be an almost proper algebraic stack with a line bundle $\cL$ that is nef on exceptional lines and suppose that $\cM^s \subset \cM$ is open. Then the stack $\cM^s$ admits a coarse moduli space $\cM^s \to M$, where $M$ is a separated algebraic space.
\end{proposition}
\begin{proof}
    By Proposition \ref{Prop:sep} we know that $\Delta \colon \cM^s\to \cM^s\times \cM^s$ is proper and we assumed it to be affine, so it is finite. Therefore by the Keel-Mori theorem \cite{KM} \cite[Theorem 1.1]{Conrad} the stack $\cM^s$ admits a coarse moduli space in the category of algebraic spaces.
    \hfill $\Box$
\end{proof}
\subsection{The example of GIT-quotients}
\begin{proposition}\label{Prop:X/GAlmostProper}
    Let $X$ be a proper scheme with an action of a reductive group $G$ and let $\cL$ be a $G$-linearized bundle that is numerically effective, then $([X/G],\cL)$ is an almost proper stack and $\cL$ is nef on exceptional lines.
\end{proposition}
\begin{proof}
    As any morphism $\Spec K \to [X/G]$ can, after passing to a finite extension $K^\prime/K$ be lifted to $X$ the stack $[X/G]$ satisfies the first part of the valuative criterion.

    Let $f\colon ST_R =\Spec R \cup_{\Spec K} \Spec R \to [X/G]$ be a morphism. Since $R$ is complete with algebraically closed residue field and $X \to [X/G]$ is smooth we can lift $j_x \circ f$ and $j_y\circ f$ to morphisms $\tilde{f}_x,\tilde{f}_y \colon \Spec R \to X$. Now since the morphism $f$ defines an isomorphim $\phi_K \colon f_x|_K \cong f_y|_K$ there exists $g_K\in G(K)$ such that $\tilde{f}_y|_K = g_K \tilde{f}_x$.

    Using the Cartan decomposition $G(K) = G(R) T(K) G(R)$ we write $g_K=k_y\lambda(\pi)k_y$ with $k_y,k_x\in G(R)$ and some cocharacter $\lambda \colon \bG_m \to G$. Replacing $\tilde{f}_x,\tilde{f}_y$ by $k_x\tilde{f}_x,k_y^{-1}\tilde{f}_x$ respectively, we obtain may assume that $\tilde{f}_y|_k = \lambda(\pi) \tilde{f}_x$, i.e. for this choice the isomorphism $\phi_K$ is defined by the element $\lambda(\pi)\in G(K)$.

     This defines a $(\bG_m,\lambda)$-equivariant morphism  $F\colon \Spec \big(R[x,y]/(xy-\pi)\big) - \{0\}  \to X$ that we can describe explicitly by  $\lambda \times f_x \colon \bG_{m,R}=\Spec \big(R[x,x^{-1}]) \to X$ and similarly by $\lambda^{-1}\times f_y$ on $\Spec \big(R[y,y^{-1}])$, which glues because $\lambda(x)f_x =\lambda(y^{-1}\pi)f_x=\lambda^{-1}(y)f_y$ on the intersection.

     This morphism is a lift of $f$, i.e. fits into a commutative diagram
    $$\xymatrix{
        \Spec \big(R[x,y]/(xy-\pi)\big) - \{0\}  \ar[r]^-F\ar[d]& X\ar[d]\\
        [\Spec \big(R[x,y]/(xy-\pi)\big) - \{0\}/\bG_m]=\ST_R \ar[r]^-f & [X/G],
    }$$
    because taking the standard sections $s_x,s_y\colon \Spec R \to \Spec \big(R[x,y]/(xy-\pi)\big) - \{0\}$ given by $x=1$ and $y=1$ the identification of $\ST_R$ with the quotient appearing in the above diagram is induced by the morphism of groupoids
    $$[\sxymat{\Spec K \dar& \Spec R \coprod \Spec R}] \to  [\Spec \big(R[x,y]/(xy-\pi)\big) - \{0\}/\bG_m]$$ given by $\pi \in \bG_m(K),s_x,s_y$ and by construction $F$ maps $\pi$ to $\lambda(\pi)=\phi_K$.

    Since $X$ is proper the morphism $F$ extends after an equivariant blow up and since $\cL$ is nef the numerical condition will automatically be satisfied.
    \hfill $\Box$
\end{proof}
\subsection{The example of $G$-bundles on curves}\label{Sec:Gbundles}
\begin{proposition}\label{Prop:BunGstar}
    Let $G$ be a reductive group and $C$ a smooth projective, geometrically connected curve. The stack $\Bun_G$ is almost proper and the line bundle $\cL_{\det}$ is nef on exceptional lines.
\end{proposition}
\begin{proof}
    We follow the same strategy as for GIT-quotients, replacing the projective atlas $X$ by the Beilinson-Drinfeld Grassmannian $p\colon \GR_G \to \Bun_G$, i.e.,
    $$\GR_G(S) = \left\{ (\cE,D,\phi) \left| { \cE\in \Bun_G(S), D\in C^{(d)}(S) \text{ for some $d$} \atop \phi \colon \cE|_{C\times S - D} \map{\cong} G \times (C\times S - D) }\right.\right\}.$$
    It is known that $\GR_G$ is the inductive limit of projective schemes, that $\cL_{\det}$ defines a line bundle on $\GR_G$ that is relatively ample with respect to the morphism to $\coprod_d C^{(d)}$ and that the forgetful map $\GR_G \to \Bun_G$ is formally smooth. Moreover this morphism admits sections locally in the flat topology (\cite[Section 5.3]{BD}, \cite{Fal}).

    To show that $\Bun_G$ is almost proper take $f\colon ST_R \to \Bun_G$. After extending $k$ we may assume that $R=k\pbl \pi\pbr $. This defines bundles $\cE_x,\cE_y$ on $C \times \Spec R$ together with an isomorphism $\cE_x|_{C\times \Spec K}\cong \cE_x|_{C\times \Spec K}$. By the properties of $\GR_G$ we can find lifts $\tilde{f}_x,\tilde{f}_y \colon  \Spec R \to \GR_G$. In particular we find a divisor $D=D_x \cup D_y\subset C_R$ such that
    $\cE_x|_{C_R-D}\cong G \times (C_R-D) \cong \cE_y|_{C_R - D}$. Through $\tilde{f}_x,\tilde{f}_y$ the isomorphism $\cE_x|_{C_K} \cong \cE_y|_{C_K}$ therefore defines an element $g\in G(C_K-D) \subset G(K(C))$. As in Proposition \ref{Prop:X/GAlmostProper}, we would like to apply Cartan decomposition now for the field $K(C)$ that comes equipped with a discrete valuation induced by the valuation of $R$, i.e. the valuation given by the codimension $1$ point $\Spec k(C) \in C_R$. Its ring of integers $\cO_K$ are meromorphic functions on $C_K$ that extend to an open subset on the special fiber $C_k$.

    Using this we can write $g=k_1 \lambda(\pi) k_2$ with $k_i\in G(\cO_K)$. Now each of the $k_i$ defines an element $k_i\in G(U)$ for some open subset $U\subset C_R$ that is dense in the special fiber, so after enlarging $D$ we may assume $k_i\in G(C_R -D)$.

    The elements $k_i$ allow us to modify the lifts $\tilde{f}_x,\tilde{f}_y$ such that with respect to these new maps we find $g=\lambda(\pi)$.

    As before this datum defines a $\bG_m,\lambda$-equivariant morphism $$\Spec(R[x,y]/(xy-\pi)) -\{0\} \to \GR_G$$ and by ind-projectivity this can be extended after a suitable blow up to a $\bG_m,\lambda$-equivariant morphism, which defines $\Bl_{I}(\oST) \to \Bun_G$. Finally, as the $\bG_m$-action preserves the forgetful map $\GR_G \to C^{(d)}$ and $\cL_{\det}$ is relatively ample with respect to this morphism we see that $\cL_{\det}$ is nef on exceptional lines.~\hfill $\Box$
\end{proof}

\section{Torsors under parahoric group schemes on curves}\label{Sec:Parahoric}
In this section we give our main application to moduli of torsors
under Bruhat-Tits group schemes on curves as introduced by Pappas
and Rapoport \cite{PR}. It will turn out that the notion of
stability we find is a variant of the one introduced by Balaji and
Seshadri in the case of generically split group schemes. We will
then apply this to construct coarse moduli spaces for stable torsors
over fields of arbitrary characteristic.




\subsection{The setup}
We fix a smooth projective, geometrically connected curve $C/k$ and
$\cG \to C$ a parahoric Bruhat-Tits group scheme in the sense of
\cite{PR}, i.e., $\cG$ is a smooth affine group scheme with
geometrically connected fibers, such that there is an open dense
subset $U\subset C$ such that $\cG|_U$ is reductive and such that
for all $p \in C-U$ the restriction $\cG|_{\Spec \cO_{C,p}}$ is a
parahoric group scheme as in \cite{BT2} (see Appendix
\ref{AppendixBT} for details). We will denote by $\Ram(\cG) \subset
C$ the finite set of closed points for which the fiber $\cG_x$ is
not a reductive group.

We will denote by $\Bun_{\cG}$ the moduli stack of $\cG$-torsors on
$C$. As usual we will often denote base extensions by an index, i.e.
for a $k$-scheme $X$ we abbreviate $X_C:= X \times C$.

\begin{example}{\rm
    It may be helpful to keep the following examples in mind:
    \begin{enumerate}
        \item[\rm (1)] (Parabolic structures) Let $G/k$ be a reductive group, $B\subset G$ be a Borel subgroup and $p_1,\dots,p_n\in C(k)$ rational points. We define $\cG_{\un{p},B}$ to be the smooth groupscheme over $C$ that comes equipped with a morphism $\cG_{\un{p},B} \to G_C$ such that for all $i$ the image
        $\cG_{\un{p},B}(\cO_{C,p_i}) = \{ g\in G(\cO_{C,p_i}) | g \mod p_i \in B \} $
        is the Iwahori subgroup. Since $\cG_{\un{p},B}(\cO_{C,p_i})$ is the subgroup of automorphism group of the trivial $G$ torsor that fixes the Borel subgroups $B\subset G \times p_i$ torsors under this group scheme are $G$-bundles equipped with a reduction to $B$ at the points $p_i$.

        \item[\rm (2)] (The unitary group) Suppose $\Char(k)\neq 2$ and let $\pi \colon \tilde{C} \to C$ is a possibly ramified $\bZ/2\bZ$-covering then the group scheme $\pi_* \GL_{n,\tilde{C}}$ admits an automorphism, given by the $()^{t,-1}$ on the group and the natural action on the coefficients $\cO_{\tilde{C}}$. The invariants with respect to this action is called  the unitary group for the covering. Torsors under this group scheme can be viewed as vector bundles on $\tilde{C}$ that under the involution become isomorphic to their dual.
    \end{enumerate}}
\end{example}

\subsection{Line bundles on $\Bun_{\cG}$}\label{PicG}

As observed by Pappas and Rapoport (\cite{PR},\cite{Unif}) there are
many natural line bundles on $\Bun_{\cG}$:

\begin{enumerate}
\item[\rm (1)] We define $\cL_{\det}$ to be the determinant line bundle given by $$\cL_{\det}|_{\cE}:= \det\big(H^*(C,\Ad(\cE)\big)^{-1}.$$
\item[\rm (2)] For every $x\in \Ram(\cG)$ we have a homomorphism $X^*(\cG_x) \incl{} \Pic(\Bun_{\cG})$ induced from the pull back via the canonical map $\Bun_{\cG} \to B \cG_x$ given by $\cE \mapsto \cE|_x$ and the canonical morphism
$$X^*(\cG_x) =\Hom(\cG_x,\bG_m) \to \Mor(B \cG_x,B\bG_m)\cong \Pic(B \cG_x),$$
which is surjective on isomorphism classes. We write $\cL_{\chi_x}$
for the line bundle corresponding to $\chi_x \in X^*(\cG_{x})$.
\item[\rm (3)] We will abbreviate $$\cL_{\det,\un{\chi}}:= \cL_{\det} \tensor \bigotimes_{x\in \Ram(\cG)} \cL_{\chi_x}.$$
\end{enumerate}

As before, positivity of $\cL$ will be checked on affine
Grassmannians. Let us fix the notation. For a point $x\in X$ we
denote by $\Gr_{\cG,x}$ the ind-projective scheme classifying
$\cG$-torsors on $X$ together with a trivialization on $C-\{x\}$.
Its $\kbar$-points are
$\cG(K_{\overline{x}})/\cG(\widehat{\cO}_{\overline{x}})$. It comes
with a forgetful map $$\glue_x \colon \Gr_{\cG,x} \to
\Bun_{\cG_x}.$$

By definition, the bundles obtained from $\glue_x$ are canonically
trivial outside $x$, so the bundles $\cL_{\chi_x}$ pull back to the
trivial line bundle on $\Gr_{\cG_y}$ for $y\neq x$.

To check that $\cL$ is nef on exceptional lines, we will need a line
bundle $\cL=\cL_{\det,\un{\chi}}$ such that for all $x\in X$ the
bundle pulls back to a positive line bundle on the corresponding
affine Grassmannian.

\begin{remark}{\rm 
    If $\cG_\eta$ is simply connected, absolutely almost simple and splits over a tame extension the positivity condition can be given explicitly, as for example computed in \cite[Section 4]{Zhu}.
    As this requires some more notation we only note that $\cL_{\det}$ always satisfies this numerical condition as this is the pull back of a determinant line bundle on a Grassmanian $\Gr_{\GL_N,x}$.}

\end{remark}

The proof of Proposition \ref{Prop:BunGstar} now applies to $\cG$,
as the proof only uses a group theoretic decomposition at the
generic point of $C$ where $\cG$ is reductive. We therefore find:

\begin{proposition}\label{Prop:BuncGstar}
    Let $\cG$ be a parahoric Bruhat--Tits group on $C$ and let $\cL_{\det,\un{\chi}}$ be chosen such that for all $x\in \Ram(\cG)$ the bundle $\glue_x^*\cL_{\det,\un{\chi}}$ is nef on $\Gr_{\cG,x}$. Then the pair $(\Bun_{\cG},\cL_{\det,\un{\chi}})$ satisfies the valuative criterion ($\star$).
\end{proposition}

To obtain coarse moduli spaces we now have to show that the stable
locus is an open subset of finite type. For this we will need
analogs of the basic results on stability for $\cG$-torsors. To do
this we first need to rephrase $\cL_{\det,\un{\chi}}$-stability in
terms of reductions of structure groups.
\subsection{Preliminaries on parabolic subgroups of Bruhat--Tits group schemes}
As before, very close degenerations of $\cG$-bundles will give us
cocharacters $\bG_{m,C} \to \Aut_{\cG/C}(\cE)=:\cG^{\cE}$. In order
to describe these in terms of reductions of structure group we first
need some general results on cocharacters and analogs parabolic
subgroups of Bruhat--Tits group schemes.

%

Let us first consider the local situation: Let $R$ be a discrete
valuation ring with fraction field $K$, $\pi\in R$ a uniformizer and
$k=R/(\pi)$ the residue field. Let  $\cG \to \Spec R$  parahoric
Bruhat-Tits group scheme.

Given $\lambda \colon \bG_{m,R} \to \cG$  we denote by
\begin{enumerate}
    \item[\rm (1)] $\cP_{\lambda}(S) := \{ g\in \cG(S) | \lim_{t\to 0} \lambda(t)g\lambda(t)^{-1} \text{ exists in }\cG(S) \}$ the concentrator scheme of the action of $\bG_m$.
    \item[\rm (2)] $\cL_{\lambda}(S) := \{ g\in \cG(S) | \lambda(t)g\lambda(t)^{-1} = g \}$ the centralizer of $\lambda$.
    \item[\rm (3)] $\cU_{\lambda}(S):=  \{ g\in \cG(S) | \lim_{t\to 0} \lambda(t)g\lambda(t)^{-1} = 1 \}$
\end{enumerate}
These are analogs of parabolic subgroups and Levi subgroups of
$\cG$. Alternatively one could define such analogs by taking the
closure of parabolics in the generic fiber $\cG_K$. The following
Lemma shows that this leads to an equivalent notion:

\begin{lemma}\label{parabolicBT}
Let $R$ be a discrete valuation ring and $\cG_R$ a parahoric
Bruhat--Tits group scheme over $\Spec R$.
\begin{enumerate}
    \item[\rm (1)] Given a 1-parameter subgroup $\lambda \colon \bG_{m,R} \to \cG$ the group $\cP_{\lambda}$ is the closure of $\cP_{K,\lambda} \subset \cG_K$, $\cL_{\lambda}$ is the closure of the Levi subgroup  $\cL_{K,\lambda} \subset \cG_K$ and $\cU_\lambda$ is the closure of the unipotent radical of $\cP_{K,\lambda}$. The group $\cL_\lambda$ is again a Bruhat--Tits group scheme.
    \item[\rm (2)] Let $P_K\subset \cG_K$ be a parabolic subgroup and denote by $\cP \subset \cG$ the closure of $P_K$ in $\cG$. Then there exists a 1-parameter subgroup $\lambda \colon \bG_{m,R} \to \cG_R$ such that $\cP=\cP_{\lambda}$.
    \item[\rm (3)] Let $P_K\subset \cG_K$ be a parabolic subgroup and $\lambda_K \colon \bG_{m,K} \to \cG_K$ be a cocharacter such that $P_K=P_{\lambda_K}$. Then there exist an element $u\in \cU(K)$ such that $\lambda^u_K:= u\lambda_K u^{-1}$ extends to a morphism $\lambda^u \colon \bG_m \to \cP$. The class of $u$ in $\cU(K)/\cU(R)$ is uniquely determined by $\lambda_K$.
\end{enumerate}
\end{lemma}
Before proving the lemma let us note that part (3) will be useful to
define a Rees construction for $\cG$-bundles. In the global setup of
a group scheme $\cG/C$ on a curve we cannot expect that every
parabolic subgroup $P_{k(C)} \subset \cG_{k(C)}$ can be defined by a
globally defined cocharacter $\bG_{m,C} \to \cG$, as $\cG$ may not
admit any non-trivial cocharacters. However, given $\lambda_{k(C)}$
part (3) will give us a canonical inner form of $\cG$ for which
$\lambda$ extends.
\begin{remark}\label{Rem:FinitelyConjP}{\rm
    In the setup of the above lemma given two parabolic subrgoups $\cP_K,\cP^\prime_K\subset \cG_K$ that are conjugate over $\cG_K$, their closures $\cP,\cP^\prime\subset \cG$ need not be conjugate. Roughly this is because the closure contains information about the relative position of the generic parabolic and the parahoric structure in the special fiber. More precisely, the Iwasawa decomposition we know that for any Borel subgroup $\cB_K\subset \cG_K$  we have $\cG(K)=\cB(K) W \cG(R)$ (\cite[Propositions 4.4.3, 7.3.1]{BT1}). This implies that there are only finitely many conjugacy classes $\cP$ of closures of generic parabolic subgroups.}
\end{remark}

\begin{proof}[Proof of Lemma \ref{parabolicBT}.]
    To show the first part of (1) we have to show that the generic fibers of $\cL_{\lambda},\cU_\lambda$ and $\cP_{\lambda}$ are dense.
    As $\cG$ is smooth over $R$ the fixed point scheme $\cL_{\lambda}$ and the concentrator scheme $\cP_{\lambda}$ are both regular (\cite[Theorem 5.8]{Hesselink} ). Let $x\in \cL_\lambda(R/(\pi))\subset \cG(R/(\pi))$ be a closed point of the special fiber then $T_{\cG,x} \to T_{\Spec R,0}$ is surjective and equivariant, so there exists an invariant tangent vector lifting the tangent direction in $0$. Thus $\cL_\lambda$ is smooth over $R$ and therefore the generic fiber is dense. The morphism $\cP_{\lambda} \to \cL_{\lambda}$ is an affine bundle (\cite[Theorem 5.8]{Hesselink}), so the generic fiber of $\cP_\lambda$ must also be dense and $\cU_\lambda$ is even an affine bundle over $\Spec R$.

    Let us prove (2) and (3). Any parabolic subgroup $P_K$ of the reductive group $\cG_K$ is of the form $P_K=P_K(\lambda)$ for some $\lambda\colon \bG_{m,K} \to \cG_K$ \cite[Lemma 15.1.2]{Springer}. The image of $\lambda$ is contained in a maximal split torus of $\cG_K$ and these are all conjugate over $K$ \cite[Theorem 15.2.6]{Springer}. Fixing a maximal split torus $\cT_R \subset \cG_R$ we therefore find $g\in G(K)$ such that $g\lambda g^{-1} \colon \bG_m \to \cG_K$ factors through $\cT$. By Iwasawa decomposition we can write $g= kwu$ with $k\in \cG(R),w\in N(\cT),p\in \cU(K)$. Thus we can conjugate $\lambda$ by an element of $\cU(K)$ such that it extends to $\cG_R$. Applying (1) to this subgroup we find (2). It also shows the existence statement in (3). To show uniqueness assume that $\lambda \colon \bG_{m,R} \to \cP_{\lambda}$ is given and that $u\in \cU(K)$ is such that $u\lambda u^{-1}$ still defines a morphism over $R$. Recall that $\cU_K$ has a canonical filtration $\cU_{K,\geq r} \subset \dots \subset \cU_{K,\geq 1} = \cU_K$ such that $\cU_{K,i}/\cU_{K,{i+}} \cong \prod_{\alpha | \alpha\circ \lambda = i} \cU_{\alpha}$. Write $\cU_i :=  \prod_{\alpha | \alpha\circ \lambda = i} \cU_{\alpha}$ and decompose $u= u_1 \cdot \dots \cdot u_r$.

    We know that $u\lambda(a) u^{-1} \in P(R)$ for all $a\in R^*$, i.e. $\cU(R) \ni u\lambda(a) u^{-1} \lambda^{-1}(a)$. The image of this element in $\cU_{\geq 1}/\cU_{\geq 2}$ is $u_1 a u_1^{-1} a^{-1}$ which can only be in $\cU_1(R)$ for all $a$ if $u_1\in \cU(R)$, but then we can replace $u$ by $u_2 \dots u_r$  and conclude by induction.

    Finally we need to show that the group scheme $\cL_\lambda$ in (1) is a Bruhat--Tits group scheme.
    By construction it suffices to show this after an \'etale base change $\Spec R^\prime \to \Spec R$, so we may assume that $\cG_R$ is quasi-split, i.e. that $\cG_K$ contains a maximal torus $\cT_K$ and that $\cG$ contains the connected N\'eron model $\cT$ of $\cT_K$.

    As conjugation by elements of $\cG(R)$ produces isomorphic group schemes we may assume as above that $\lambda\colon \bG_{m,R} \to \cT \subset \cG$. The scheme $\cG_R$ is given by a valued root system and the restriction of this to the roots of $\cL_{\lambda,K}$ defines a Bruhat--Tits group scheme with generic fiber $\cL_{\lambda,K}$, contained in $\cL_\lambda$. Finally, by definition  the special fiber of $\cL_{\lambda}$ is the centralizer of a torus, so it is connected. Thus the smooth scheme $\cL_\lambda$ has to be equal to this Bruhat--Tits scheme.
\hfill $\Box$
\end{proof}

Let us translate this back to our global situation: As before let
$\cG/C$ be a Bruhat--Tits group scheme over our curve $C$ and denote
by $\eta\in C$ the generic point of $C$.

\begin{lemma}\label{ReesBT}
        Let  $\lambda_{\eta} \colon \bG_{m,\eta} \to \cG_\eta$ be a cocharacter and $\cP_{\eta,\lambda}$ the corresponding parabolic, $\cU_{\eta,\lambda}$ its unipotent radical and
$\cU_\lambda,\cP_\lambda$ the closures of
$\cU_{\eta,\lambda},\cP_{\eta,\lambda}$ in $\cG$. Then:
        \begin{enumerate}
            \item[\rm (1)] $\cP_\lambda,\cU_\lambda$ are smooth group schemes over $C$. The quotient $\cP_\lambda/\cU_\lambda=:\cL_\lambda$ is a Bruhat--Tits group scheme.
            \item[\rm (2)] The morphism $\lambda_\eta$ extends to $\overline{\lambda}\colon \bG_m \times C \to Z(\cL_\lambda)\subset \cL_\lambda$.
            \item[\rm (3)] There exist a canonical $ \cU_\lambda$-torsor $\cU^u$ together with an isomorphism $\cU^u|_{k(C)} \cong \cU_{k(C)}$ such that $\lambda$ extends to $\lambda \colon \bG_m \to \cU^u \times^{\cU,\conj} \cG$.
        \end{enumerate}
\end{lemma}
\begin{proof}
        First note that $\lambda_{\eta}$ extends canonically to a Zariski open subset $U\subset C$ (e.g. \cite[Expos\'e XI, Proposition 3.12 (2)]{SGA3}). Let us denote this morphism $\lambda_U \colon \bG_{m,U} \to \cG_U$. Over $U$ the subgroup $\cP_{\lambda,U} \subset \cG_U$ is a parabolic subgroup of the reductive group scheme $\cG_U$, the groups $\cU_{\lambda,U},\cL_{\lambda,U}$ are the corresponding unipotent radical and Levi subgroup.
        Part (1) and (2) can thus be checked locally around all points $x\in C-U$ and there Lemma \ref{parabolicBT} gives the result. Since a $\cU$-torsor together with an isomorphism $\cU^u|_{k(C)} \cong \cU_{k(C)}$ is given by a finite collection of elements $u_x \in U(k(C))/U(\cO_{C,x})$ for some $x\in C$ the last part also follows from Lemma \ref{parabolicBT} (3).
        \hfill $\Box$
\end{proof}

This lemma allows us to generalize the Rees construction: Given
$\lambda_{\eta}\colon \bG_{m,\eta} \to \cG_{\eta}$ the above lemma
constructs an inner form $\cP^u \subset \cG^u$ such that
$\lambda\colon \bG_{m,C} \to \cP^u$ extends. As we proved that
fiberwise conjugation by $\lambda$ contracts $\cP^u$ to
$\cL_\lambda$ we again obtain the morphism of group schemes over
$C\times \bG_m$
\begin{align*}
\conj_\lambda\colon \cP^u \times \bG_m &\to P^u \times \bG_m\\
(p,t) &\mapsto (\lambda(t)p\lambda(t)^{-1},t).
\end{align*}
By \cite[Proposition 4.2]{Hesselink} this homomorphism extends to
\begin{align*}
\gr_\lambda\colon \cP^u  \times \bA^1 &\to \cP^u  \times \bA^1\\
\end{align*}
in such a way that $\gr(p,0)= \lim_{t\to 0}
\lambda(t)p\lambda(t)^{-1} \in \cL_\lambda \times 0$.

Thus given a $\cP^u$ torsor $\cE^u$ we can define the Rees
construction:

$$\Rees(\cE^u,\lambda):=  (\cE^u \times \bA^1) \times^{\gr_{\lambda}} (\cP^u  \times \bA^1).$$

Given a $\cG$ torsor $\cE$ together with a reduction $\cE_{\cP}$ to
a parabolic subgroup $\cP$ and a cocharacter $\lambda_\eta \colon
\bG_m \to \cP_{\eta}$ defining $\cP$ we take the associated
$\cP^u$-torsor $\cE^u:=\Isom_{\cP/C}(\cE_{\cP},\cU^u \times^\cU
\cP)$ and define
$$\Rees(\cE_{\cP},\lambda_\eta):= \Isom_{\cP^u/C}(\Rees(\cE^u,\lambda),\cU^u \times^\cU \cP).$$

As before, this construction only depends on $\cP$ and the
composition $\overline{\lambda}\colon \bG_{m,\eta} \to \cL_{\eta}$.

\begin{remark}\label{Rem:ComputingWtcG}{\rm
Note that the adjoint bundle of $\Rees(\cE_{\cP},\lambda_{\eta})$ is
the vector bundle on $C \times [\bA^1/\bG_m]$ that on the generic
fiber corresponds to the very close degeneration given by
$\Ad(\cE_{\cP})_{\eta}$ and the cocharacter given by $\Ad(\lambda)$.
As we know that this already defines the very close degeneration,
the passage to $\cP^u$-torsors was only needed in order to give a
formula for the very close degeneration in terms of $\cG$ torsors.}
\end{remark}

\subsection{Very close degenerations of $\cG$-bundles}
We can now classify very close degenerations $f\colon [\bA^1/\bG_m]
\to \Bun_\cG$. Such a morphism defines a $\cG$-torsor $\cE$ over
$C\times [\bA^1/\bG_m]$ and $f(0)$ defines a non-trivial cocharacter
$\lambda\colon \bG_m \times C \to \Aut_{\cG/C}(\cE_0) =:
\cG^{\cE_0}$. This defines $\cL_\lambda,\cP_\lambda \subset
\cG^{\cE_0}$. As $\cL_\lambda$ is the concentrator scheme of the
$\bG_m$ action we again get that the morphism:
 \begin{align*}
 \conj_\lambda\colon \cP_\lambda \times \bG_m &\to \cP_\lambda \times \bG_m\\
 (p,t) &\mapsto (\lambda(t)p\lambda(t)^{-1},t)
 \end{align*}
 extends to
 \begin{align*}
 \gr_\lambda\colon \cP_\lambda \times \bA^1 &\to \cP_\lambda \times \bA^1
 \end{align*}
 in such a way that $\gr(p,0)= \lim_{t\to 0} \lambda(t)p\lambda(t)^{-1} \in \cL_\lambda \times 0$.
And so we can define the Rees construction for $\cP_\lambda$-bundles
$\cF_\lambda$:
$$\Rees(\cF_{\lambda},\lambda) := [\big((\cF_{\lambda} \times \bA^1)\times^{\gr_\lambda} (\cP_\lambda \times \bA^1)\big)/\bG_m].$$

The $\cG$ torsor $\cE$ is determined by the $\cG^{\cE_0}$-torsor
$\cE^\prime := \Isom(\cE,\cE_0)$.  As before $\Isom(\cE_0,\cE_0)$
being the trivial $\cG^{\cE_0}$ torsor $\cE^\prime_0$ comes equipped
with a canonical reduction to $\cL_\lambda$ and as in Lemma
\ref{equivG} the corresponding reduction to $\cP_\lambda$ lifts
canonically to $C\times [\bA^1/\bG_m]$. We denote this reduction by
$\cE^\prime_{\lambda}$. Now we can argue as in Lemma \ref{Lem:ReesG}
to identify
$$\cE^\prime_{\lambda} \cong \Rees(\cE_{\lambda,1}^\prime,\lambda)$$
and thus
$$\cE \cong \Isom_{\cG^{\cE_0}}(\cE^\prime,\cE_0) \cong \Isom_{\cG^{\cE_0}}(\Rees(\cE^\prime_{\lambda,1},\lambda)\times^{\cP_\lambda}\cG^{\cE_0},\cE_0).$$

Formulating stability in these terms would have the annoying aspect
that all possible group schemes $\cG^{\cE_0}$ would appear in the
formulation. This can be avoided this by restricting to the generic
fiber as follows:

By the description of $G$-bundles over $[\bA^1/\bG_m]$ (Lemma
\ref{equivG}) we have
$$\Isom(\cE,\cE_0)_{\eta\times [\bA^1/\bG_m]}=\cE^\prime|_{\eta \times [\bA^1/\bG_m]} \cong [\bA^1 \times \cG^{\cE_0}_{\eta}/\bG_m].$$
In particular $\cE_{1,\eta} \cong \cE_{0,\eta}$. Therefore the
canonical reduction of $\cE_1^\prime$ to $\cP_{\lambda}$ corresponds
to a reduction of $\cE_{1,k(C)}$ to a parabolic $P_{\lambda_K}
\subset \cG^{\cE_1}$.

Thus given the cocharacter $\lambda \colon \bG_m \to \cG^{\cE_0}$ we
can find a cocharacter $\lambda_K \colon \bG_m \to \cG^{\cE_1}$ such
that the canonical reduction of $\cE_1^\prime$ to
$\cP_{\lambda}\subset \cG^{\cE_0}$ defines a reduction
$\cE_{\cP^\prime_\lambda}$ of $\cE$ to $\cP_{\lambda}^\prime:=
\overline{\cP_{\lambda_K}} \subset \cG^{\cE_1}$.

Given the reduction of  $\cE$ to $\cP^\prime_\lambda$ we already
defined the corresponding Rees construction. Thus we find:
\begin{lemma}\label{Lem:vcdGbundlesintrinsic}
    Let $\cG \to C$ be a Bruhat--Tits group scheme, $\cE\in \Bun_{\cG}$ and $\cG^{\cE} := \Aut_{\cG/C}(\cE)$ the corresponding inner form of $\cG$.

    Then any morphism $f\colon [\bA/\bG_m] \to \Bun_{\cG}$ with $f(1) = \cE$ can be obtained from the Rees construction applied to a generic cocharacter $\lambda_{\eta} \colon \bG_m \to \cG^{\cE}_{k(C)}$.
\end{lemma}

In the above we described reductions by cocharacters $\lambda \colon
\bG_m \to \cG^{\cE}_{k(C)}$ because this description works over any
field. As in \cite{Behrend} this is more suitable to study
rationality problems for canonical reductions of $\cG$-bundles. Over
algebraically closed fields we can also reformulate this in terms of
reductions to parabolic subgroups of $\cG_{k(C)}$ as follows:

\begin{lemma}\label{Lem:vcdGbundlesClassical}
    Let $k=\overline{k}$ be an algebraically closed field, $\cG \to C$ be a Bruhat--Tits group scheme, $\cE\in \Bun_{\cG}(k)$.
    Then any morphism $f\colon [\bA/\bG_m] \to \Bun_{\cG}$ with $f(1) = \cE$ can be obtained from the Rees construction applied to a reduction of $\cE$ to a subgroup $\cP_\lambda \subset \cG$ which is defined by a generic cocharacter $\lambda \colon \bG_m \to \cG_{k(C)}$.
\end{lemma}
\begin{proof}
By Lemma \ref{Lem:vcdGbundlesintrinsic} we know that $f$ can be
defined by applying the Rees construction to a cocharacter
$\lambda_\eta\colon \bG_m \to \cG_{k(C)}^{\cE}$. Let
$\cP_\lambda^{\cE} \subset \cG^{\cE}$ be the closure of the
parabolic subgroup defined by $\lambda_\eta$. As explained in the
appendix (Proposition \ref{Prop:WeakApproximiation}) any $\cG$
bundle on $C_{\overline{k}}$ can be trivialized over an open subset
$U \subset C_{\kbar}$ which contains $\Ram(\cG) \subset U$. Choose
such a trivialization $\psi\colon \cE|_{U} \cong \cG|_{U}$. This
induces an isomorphism $\cG_{k(C)}^{\cE} \cong \cG_{k(C)}$ and this
defines $\lambda \colon \bG_m \to \cG_{k(C)}$. Denote by
$\cP_\lambda\subset \cG$ the closure of the parabolic subgroup
defined by $\lambda$. By construction we know that
$\cP_\lambda^\cE|_U \cong \cP_\lambda|_U$ and since $\cG$ is a
reductive group scheme over $C-U$ the groups $\cP_\lambda$ and
$\cP_\lambda^{\cE}$ are also isomorphic in a neighborhood of the
points in $C-U$. Thus the reduction of $\cE$ to $\cP_\lambda^{\cE}$
defines a reduction of $\cE$ to $\cP_\lambda$. This proves our
claim.
\hfill $\Box$
\end{proof}

\begin{remark}\label{Rem:GloballyFinitelyManyConjClassesOfP}{\rm
    In the case of $G$-bundles there are only finitely many conjugacy classes of parabolic subgroups $P\subset G$ and therefore any very close degeneration of a bundles $\cE$ is induced from viewing $\cE$ as lying in the image of a morphism $\Bun_P \to \Bun_G$. To show that the semistable points of $\Bun_G$ form an open substack the fact that one needs only to consider finitely many such $P$ is helpful.

    From Lemma \ref{Lem:vcdGbundlesClassical} we can now conclude that the analogous result also holds for $\cG$-bundles: Suppose $\cP,\cP^\prime \subset \cG$ are closures of parabolic subgroups in the generic fiber. If $\cP$ and $\cP^\prime$ happen to be conjugate in $\cG$ at the generic point of $C$ they are also conjugate locally on $C-\Ram(\cG)$ as parabolic subgroups of the same type are conjugate in reductive groups. Also locally around any point $x\in \Ram(\cG)$ we saw in Remark \ref{Rem:FinitelyConjP} that there are only finitely many conjugacy classes of closures of parabolic subgroups $\cP\subset \cG$. Thus up to local conjugation in $\cG$ there are only finitely many closures of parabolic subgroups $\cP\subset \cG$.

    As in the case of $G$-bundles, if  $\cP,\cP^\prime \subset \cG$ are locally conjugate over $C$, then the transporter $\Transp_{\cG}(\cP,\cP^\prime)$ of elements of $\cG$ that conjugate $\cP$ into $\cP^\prime$ is a $\cP-\cP^\prime$ bi-torsor (because parabolic subgroups are equal to their normalizer over $C-\Ram(\cG)$ and a local section of $\cG$ normalizes $\cP$ if and only it normalizes the generic fiber). This defines a commutative diagram
    $$ \xymatrix{  \Bun_{\cP} \ar[rr]^{\cong}\ar[dr] && \Bun_{\cP^\prime} \ar[dl] \\ & \Bun_\cG }$$
    identifying reductions to the structure groups $\cP$ and $\cP^\prime$.}
\end{remark}

\subsection{The stability condition}\label{sec:stabcond}

 Let us fix our group scheme $\cG$ and a line bundle $\cL := \cL_{\det,\un{\chi}}$ as in Section \ref{PicG}.
 In this section we want to describe $\cL$-stability condition in terms of degrees of bundles.
 Lemma \ref{Lem:vcdGbundlesClassical} and the definition of $\cL$-stability (\ref{Def:Lstab}) imply:
\begin{remark}{\rm
 A $\cG$-bundle $\cE$ on $C_{\overline{k}}$ is $\cL$-stable if and only if for all parabolic subgroups $P\subset \cG_{\overline{k}(C)}$ with closure $\cP\subset \cG$, all dominant cocharacters $\lambda\colon \bG_m \to P$ and all reductions of $\cE$ to $\cP$ we have $$\wt_{\cL}(\Rees(\cE_\cP,\lambda)) > 0.$$}
\end{remark}

As in the case of $G$-bundles we want to express the above weight in
terms degrees of line bundles attached to reductions of $\cE$. We
start out with the intrinsic formulation of reductions as in Lemma
\ref{Lem:vcdGbundlesintrinsic}, but in the end this reduces to a
computation on the adjoint bundle $\Ad(\cE)$.

 Fix $\cE$ a $\cG$-bundle, $\cE_{\cB}$ a reduction to a Borel subgroup $\cB \subset \Aut_{\cG}(\cE)$. As before let us denote by $\cU\subset \cB$ the closure of the  unipotent radical over the generic point $\eta\in C$ and $\cT=\cB/\cU$ the maximal torus quotient. Fix $S\subset \cT_{\eta}\subset \cB_\eta$ a maximal split torus in a lifting of the maximal torus at $\eta$. We will denote by $\Phi = \Phi(\cG_{\eta},S)$ the roots of $\cG_\eta$ and $\Phi_{\cB}^+$ will be the roots that are positive with respect to $\cB$.

For any point $x\in \Ram(\cG)$ we obtain a character
 $\chi_x^{\cB} \colon \cB_x \to \bG_m$ as composition $$\chi_x^{\cB}\colon \cB_x^\cE \to \cG_x^\cE \map{\chi} \bG_m.$$
 This morphism factors through $\cT_x^{\cE}$. For any $\lambda\colon \bG_m \to \cT$ we will write
 $\langle \chi_x^{\cB}, \lambda \rangle := \langle \chi_x^{\cB}, \lambda|_{x} \rangle$.

 The Rees construction applied to a generic dominant 1-parameter subgroup gives us a bundle $\cE_0$, that is induced from the $\cT$ bundle $\cE_{\cB}/\cU =: \cE_{\cT}$.

 To compute the weights,  we can decompose the adjoint bundle of $\cE_0$ into weight spaces:
 $$\ad(\cE_0)=\ad(\cE_{\cT}) \oplus_{a\in \Phi(\cG_{\eta})} \cu_a^{\cE_0}$$ and the Rees construction induces a filtration of $\ad(\cE)$ such that the associated graded pieces are $\cu_a^{\cE}\cong \cu_a^{\cE_0}$.

 For any 1-parameter subgroup $\lambda\colon \bG_m \to \cG_{\eta}^{\cE}$ which is dominant with respect to $\cB$ we then have:
 $$\wt_{\cE_{\cB}}(\lambda):= \wt_{\cE}(\lambda) = \sum_{a\in \Phi} \chi(\cu_a^{\cE}) \langle a, \lambda\rangle + \sum_{x\in \Ram(\cG)} \langle \chi_x^{\cB} ,\lambda \rangle.$$

 As $\rk(\cu_a)=\rk(\cu_{-a})$ we can further compute:
 \begin{align*}
 \wt_{\cE_{\cB}}(\lambda) &= \sum_{a\in \Phi} \chi(\cu_a^{\cE}) \langle a,\lambda\rangle + \sum_{x\in \Ram(\cG)} \langle \chi_x^{\cB},\lambda_x \rangle \\
 &= \sum_{a\in \Phi} \deg(\cu_a^{\cE}) \langle a,\lambda\rangle + \sum_{x\in \Ram(\cG)} \langle \chi_x^{\cB},\lambda_x \rangle
 \end{align*}

 In the case of unramified, constant group schemes, the degree $\deg(\cu_a)$ is a linear function in the root $a$. This does no longer hold for parahoric group schemes, but we still have relations:

 At the generic point $\eta$ of $C$ we get a canonical isomorphism $$k_a \colon \cu_{-a,\eta}^{\cE} \map{\cong} (\cu_{a,\eta}^{\cE})^\vee$$
 from the Killing form and this extends to an isomorphism at all points $c\in C$ where $\cG_c$ is reductive. Therefore, the determinant of $k_a$ defines a divisor $$D_a = \sum_{x \in \Ram(\cG)} f_{a,x}^{\cB} x.$$
With this notation we have
$$ \deg(\cu_{-a}^{\cE}) = - \deg(\cu_a^{\cE}) - \sum_{x\in \Ram(\cG)} f_{a,x}^{\cB}.$$
 Here we denote the coefficients by the letter $f$ because for a global torus $\cT \subset \cG$ with valuated root systems $f_{a,x}$ (see Section \ref{AppendixBT}) these numbers are $-(f_{a,x}+f_{-a,x})$. As any two tori are conjugate, we always find $$\left|\frac{f_{a,x}^{\cB}}{\rk(\cu_a)}\right|\leq 1.$$

 Thus we find
 \begin{align}
 \wt_{\cE_{\cB}}(\lambda) &= \sum_{a\in \Phi} \deg(\cu_a^{\cE}) \langle a,\lambda\rangle + \sum_{x\in \Ram(\cG)} \langle \chi_x^{\cB},\lambda \rangle \nonumber\\
 &= 2 \sum_{a\in \Phi^+} \deg(\cu_a^{\cE}) \langle a,\lambda\rangle +  \sum_{x\in \Ram(\cG)} \langle \chi_x^{\cB}+ \sum_{a\in \Phi^+_{\cB}} f_{a,x}^{\cB} a,\lambda \rangle \label{wtEquation}
 \end{align}

To compare this to the usual (parabolic)-degree let us fix a norm on
the set of all 1-parameter subgroups. A convenient choice for us
will be to fix for any maximal torus containing a maximal split
torus  $S_{\eta} \subset T_\eta \subset \cG_{\eta}$ the canonical
invariant bilinear form on $X_*(T_{\overline{\eta}})$:
  $$ (,):=\sum_{\alpha \in \Phi(\cG_{\overline{\eta}})} \langle \quad, \alpha \rangle \langle \quad , \alpha \rangle $$

We will denote the restriction of $(,)$ to $X_*(T)$ by the same
symbol and we will denote by $||\cdot||:=\sqrt{(\cdot,\cdot)}$ the
induced norm.

  \begin{remark}{\rm
    The bilinear form on $X_*(T_{\overline{\eta}})$ also induces a form on the cocharacter groups $X_*(\cT_{x})$ for all $x\in C$ closed.}
  \end{remark}

 \begin{remark}{\rm
    \begin{enumerate}
        \item[]
        \item[\rm (1)] If $\cG_0 = G\times X$  is a split group scheme we remarked above that $\deg(\cu_a^{\cE})$ is a linear function in $a$, denoted $\un{\deg}$ as this is the usual degree of the $\cT$-bundle $\cE_{\cT} = \cE_{\cB}/\cU$. Then the above formula reads
        $$  \wt_{\cE_{\cB}}(\lambda) =  (\un{\deg}, \lambda).$$
        This expresses the weight in terms of the degree that is classically used to define stability for $G$-torsors.
        \item[\rm (2)] If $\cG \to \cG_0 = G\times X$ is $\cG$ is obtained as the parahoric subgroup defined by the choice of parabolic subgroups $P_x \subset \cG_{0,x} =G$ the numbers $f_{a,x}$ are $0,1,-1$ depending on the relative position of $\cB_x$ and the image of $\cG^{\cE}_x \to (\cG_0^{\cE})_x$.  By the previous point this then again gives the same relation of the weight and the parabolic degree.
        \item[\rm (3)] As the difference $\deg(\cu_a^{\cE})-\deg(\cu_{-a}^{\cE})$ is always an integer, the formula also shows that as in the case of parabolic bundles the weight cannot be $0$ if for at least one $x\in \Ram(\cG)$ the group scheme $\cG|_{\cO_x}$ is an Iwahori group scheme and $\chi_x$ is chosen generically, e.g. such that  the numbers $\langle \chi_x, \check{\alpha}_i \rangle$ for some basis $\check{\alpha}_i$ of one parameter subgroups of $\cT$ are rational numbers with sufficiently large denominators.
        \item[\rm (4)] For generically split groups $\cG$ in \cite[Definition 6.3.4]{BS} Balaji and Seshadri describe stability in terms of $(\Gamma,G)$ bundles. In their setup the reductions $\cE_\cP$ are computed from taking invariants of an invariant parabolic reduction of a $G$-bundle on a covering $\tilde{C}\to C$ (\cite[Proof of Proposition 6.3.1]{BS}). The numerical invariants defining stability are then expressed as the parabolic degrees defined by characters of $P$ attached to a reduction of $\cE_\cP$. To compare this with our condition one can proceed as in (1) by using (\ref{wtEquation}). In this case, as $\cG$ is generically split, all $\cu_a^{\cE}$ are line bundles, which are the invariant direct images of the corresponding bundles of the reduction to $P$ on $\tilde{C}$. The parabolic weight then turns out to be related to the contribution form the ramification points which depends only on the relative position of the generic parabolic and the valuation of the root system defining $\cG$. As the precise relation requires a careful recollection on the relation of $(\Gamma,G)$-bundles and valued root systems we leave this to a later time.
    \end{enumerate}}
 \end{remark}

\subsection{Canonical reduction for $\cG$-torsors}
In this section we want to check that the canonical reduction of
$G$-bundles introduced by Behrend also exists for $\cG$-bundles. We
will then use this to deduce that the stack of stable and semistable
$\cG$-bundles are open substacks of finite type of $\Bun_{\cG}$.

In \cite{DHL} Halpern-Leistner gives general criteria for the
existence and uniqueness of canonical reductions for
$\theta$-reductive stacks. Unfortunately, $\Bun_{\cG}$ does not
satisfy this condition, so that we have to give a separate argument.
It will turn out that once we formulate the classical approach for
$G$-bundles (see \cite{Harder},\cite{Behrend}) in a suitable way,
most of the arguments generalize to this framework. This was also
explained by Gaitsgory and Lurie in \cite[Section
10]{GaitsgoryLurie} for a notion of stability induced from
$G$-bundles. Note that Harder and Stuhler also introduced the
concept of canonical reductions for Bruhat--Tits groups in the
adelic description of the points of the moduli stack
\cite{HarderStuhler}.

Let us fix our group scheme $\cG$ and a line bundle $\cL :=
\cL_{\det,\un{\chi}}$ as in Section \ref{PicG}. In general, to
define canonical destabilizing 1-parameter subgroups one needs to
fix a norm on the set of all such subgroups. We will simply use the
invariant form $(,)$ on $X_*(T_{\overline{\eta}})$ from the previous
section.

As for parabolic bundles, we will need the following assumption on
$\un{\chi}$. We will call $\cL$ {\em admissible} if $2 \langle
\chi_x, \check{a} \rangle \leq \rk{\cu}_a \langle a, \check{a}
\rangle $ for all roots $a$ of $\cG_x$.

As in \cite{DHL} a canonical reduction of a $\cG$-torsor should be a
1-parameter subgroup $\lambda \in X_*(\cG^{\cE}_\eta)$ such that
$\frac{\wt(\lambda)}{||\lambda||}$ is maximal. First of all this
number is bounded:


\begin{lemma}
    For every $\cG$ torsor $\cE$ there exists $c_{\cE}>0$ such that $\frac{\wt(\lambda)}{||\lambda||}\leq c$ for all $\lambda\colon \bG_{m,\eta} \to \cG^{\cE}_{\eta}$.
\end{lemma}
\begin{proof}
  As in the classical case for every reduction $\cE_{\cP}$ of $\cE$ to a parabolic subgroup we have that $H^0(\ad(\cE_{\cP})) \subset H^0(\ad(\cE))$.
  By Riemann-Roch this implies that the degree of the unipotent radical $\deg( \cE_{\cP} \times^{\cP} \cu_{\cP})$ is uniformly bounded above for all reductions.
  In turn this gives for every reduction to a Borel subgroup $\cB \subset \cG^{\cE}$ an upper bound for the weight $\wt(\omega_i)$ for all dominant coweights $\omega_i$. As any dominant $\lambda$ is a positive linear combination of these, this gives the required bound.
  \hfill $\Box$
\end{proof}

\noindent {\bf Notation.}
    For any $\cG$-torsor $\cE$ over an algebraically closed field we define $$\mu_{\max}(\cE) := \sup_{\lambda} \frac{\wt_{\cE}(\lambda)}{||\lambda||}.$$

\subsubsection{Comparing weights of different reductions}
To characterize the canonical reduction we need to compare the
weights of different reductions.

Suppose $\cB,\cB^\prime \subset \cG^{\cE}$ are two Borel subgroups.
Any two such subgroups share a maximal split torus $S\subset
\cT_{\eta} \subset \cB_{\eta}\cap \cB^\prime_{\eta}$
(\cite[Proposition 4.4]{BorelTits}).

Let $\{a_1,\dots a_n\}$ be the positive simple roots of $\cB$.

We say that $\cB,\cB^\prime$ are neighboring reductions if there
exists $i_0$ such that $-a_{i_0}$ is a positive simple root of
$\cB^\prime$ and $a_i$ are positive roots of $\cB^\prime$ for all
$i\neq i_0$.

In this case $\cB,\cB^\prime$ generate a parabolic $\cP_{1}$ that is
minimal among the parabolic subgroups that are not Borel subgroups.
Denote by $\cL_1$ the corresponding Levi quotient.

\begin{lemma}
    Let $\cB,\cB^\prime$ be neighboring parabolic subgroups. Then we have:
    \begin{enumerate}
        \item[\rm (1)] $\wt_{\cB}(\lambda)=\wt_{\cB^\prime}(\lambda)$ for all $\lambda \in Z(\cL_I)$.
        \item[\rm (2)] $\wt_{\cB}(\check{\alpha}_i) + \wt_{\cB^\prime}(-\check{\alpha}_i)\leq 0$
    \end{enumerate}

\end{lemma}
\begin{proof}
    As the weight only depends on the Rees construction for $\cE$ we may replace $\cE$ by the Rees construction applied to any $\lambda$ that is dominant for $\cB$ and $\cB^\prime$. Thus we may assume that $\cE = \cE_{\cL_1} \times^{\cL_1} \cG^{\cE}$ is induced from a $\cL_1$ torsor.
    In this case (1) is immediate from the definition, as the Rees construction for $\lambda \in Z(L_1)$ does not change $\cE_{\cL_1}$.

    Moreover, $\cL_1$ is of semisimple rank $1$ and $\cB_1:=\cB\cap\cL_1$ and $\cB^\prime_1\cap \cL_1$ are Borel subgroups that are opposite over the generic point $\eta\in C$. Let us denote by $\cu,\cu^\prime$ the Lie algebras of the corresponding unipotent radicals in $\cL_1$ and by $\cu^-:= \Lie(\cL_1)/ \Lie(\cB_1)$ and similarly $\cu^{\prime,-}$. Then we get an injective homomorphism
    $$ \cE_{\cB_1} \times^{\cB_1} \cu \to \ad(\cE_{L_1}) \to \ad(\cE_{L_1})/\ad(\cE_{\cB^\prime_1})\cong \cE_{\cB^\prime_1} \times^{\cB^\prime_1} \cu^{-,\prime}.$$
    If $a$ is a multipliable root then the homomorphism respects the filtration on $\cu,\cu^{\prime,-}$ given by the roots $a,2a$.
    So we find:
    \begin{align*}
     \deg(\cu_{a}^{\cB}) &\leq \deg(\cu_{a}^{\cB^\prime}) = - \deg(\cu_{-a}^{\cB^\prime}) -  \sum_{x\in \Ram(\cG)} f_{-a,x}^{\cB^\prime} \\
     -  \deg(\cu_{a}^{\cB}) -\sum_{x\in \Ram(\cG)} f_{a,x}^{\cB} = \deg(\cu_{-a}^{\cB}) &\geq \deg(\cu_{-a}^{\cB^\prime})
    \end{align*}
    an the same hods for $\cu_{2a}$ if $2a$ is also a root.

Thus:
$$ 2(\deg(\cu_a^{\cB}) + \deg(\cu_{-a}^{\cB^\prime}) ) \leq   \sum_{x\in \Ram(\cG)} f_{-a,x}^{\cB^\prime} + f_{a,x}^{\cB}.$$
Moreover, the map $\cu_a^{\cB} \to \cu_{a}^{\cB^\prime}$ is an
isomorphism at $x$ if and only if the groups $\cB_x^{\cE}$ and
$\cB_x^{\prime,\cE}$ are opposite in $\cG^{\cE}_x$. And in this case
$\langle \chi_x^{\cB_1},\check{\alpha}_{i_0} \rangle = - \langle
\chi_x^{\cB_1^\prime},-\check{\alpha}_{i_0} \rangle $. If the
morphism is not an isomorphism, then the Borel subgroups are
parallel in $x$, so that $\langle
\chi_x^{\cB_1},\check{\alpha}_{i_0} \rangle = \langle
\chi_x^{\cB_1^\prime},-\check{\alpha}_{i_0} \rangle $. Thus if $2
\langle \chi_x, \check{a} \rangle \leq \rk{\cu}_a \langle a,
\check{a} \rangle $ for all roots $a$ we find:
$$ 2(\deg(\cu_a^{\cB}) + \deg(\cu_{-a}^{\cB^\prime})) \langle a, \check{a} \rangle  \leq  \langle  \sum_{x\in \Ram(\cG)} f_{-a,x}^{\cB^\prime} + f_{a,x}^{\cB} + \chi_x^{\cB_1}-\chi_x^{\cB_1^\prime} , \check{a} \rangle .$$
And this means:

\begin{align*}
\wt_{\cB_1}(\check{\alpha}_{i_0}) +
\wt_{\cB_1^\prime}(-\check{\alpha}_{i_0}) & \leq 0
\end{align*}
if $\un{\chi}$ is admissible.

To compute $\wt_{\cB}(\check{\alpha}_{i_0}) +
\wt_{\cB^\prime}(-\check{\alpha}_{i_0})$ we note that $\cE_{\cP_1}
\cong \cE_{\cB_1} \times^{\cB_1}  \cP_1 \cong \cE_{\cB_1^\prime}
\times^{\cB^\prime_1} \cP_1$ and the unipotent radical of
$\Lie(\cP_1)$ has a filtration by $\cL$-invariant subspaces such
that over the generic point the associated graded pieces are
isomorphic to the unipotent groups $\cu_{b*,\eta} = \oplus_{c=b+na}
\cu_{c,\eta}^{\cB} = \oplus_{c=b-na} \cu_{c,\eta}^{\cB^\prime}$.
Moreover since $a$ is a positive root for $\cB_1$ and a negative
root for $\cB^\prime_1$ over $C$ the isomorphic bundles $\cE_{\cB_1}
\times^{\cB_1} \cu_{b*}$ and $\cE_{\cB_1^\prime}
\times^{\cB_1^\prime} \cu_{b*}$ come with canonical filtrations that
are opposite at the generic point. Therefore we find again that
$$ \sum_{c=b+na} (\deg( \cu_{c,\eta}^{\cB}) - \deg( \cu_{c,\eta}^{\cB^\prime})) \langle c, \check{\alpha}_{i_0} \rangle \leq 0$$
Summing over all $b$ we obtain the result.
\hfill $\Box$
\end{proof}

We define $\cL-\deg (\cE_{\cB})$ by the formula
$$ (\cL-\deg,\lambda) := - \wt_{\cE_{\cB}}(\lambda).$$
Then the previous lemma shows that $\cL-\deg$ defines a
complementary polyhedron as defined by Behrend \cite{Behrend}. As in
\cite[Section 4.3]{HS} we therefore obtain:

\begin{proposition}
    Suppose $\cL_{\det,\un{\chi}}$ is an admissible line bundle and $\cE$ a $\cG$-torsor. Then there exists a reduction $\lambda \colon \bG_m \to \cG_{\eta}^{\cE}$ such that $\cE_{\cP_\lambda}$ is a reduction for which $\frac{\wt(\lambda)}{||\lambda||}$ is maximal and such that for every other such reduction to a parabolic subgroup $\cQ_{\lambda^\prime}$ we have $\cQ_{\lambda^\prime} \subset \cP_{\lambda}$.
\end{proposition}

\begin{lemma}[Semicontinuity of instability]\label{lem:semicont}
    Let $\cL_{\det,\un{\chi}}$ be an admissible line bundle on $\Bun_{\cG}$.
    Let $R/k$ be a discrete valuation ring with fraction field $K$ and residue field $\kappa$ and let  $\cE_R$ be a $\cG$-torsor on $C_R$.
    \begin{enumerate}
        \item[\rm (1)] If $\cE_K$ is unstable and the canonical reduction is defined over $K$, then $\cE_{\kappa}$ is unstable and
        $$\mu_{max}(\cE_K) \leq \mu_{max}(\cE_{\kappa}).$$
        The equality is strict, unless the canonical reduction of $\cE_K$ extends to $R$.
        \item[\rm (2)]  If $\cE_K$ is semistable but not stable then $\cE_{\kappa}$ is also not stable.
    \end{enumerate}
\end{lemma}
\begin{proof} The first part follows as in \cite[Lemma 4.4.2]{HS}.
    This also shows that if $\cE_K$ admits a reduction of weight $0$, then $\cE_{\kappa}$ also cannot be stable.

    Finally suppose $\dim \Aut(\cE_{K})>0$. We know that $\cG^{\cE_R}\map{p} C_R$ is an affine group scheme of finite type over $C_R$. The group of global automorphisms of this group scheme is $\Spec p_*(\cO_{\cG^{\cE_R}})$, so the generic fiber of this is not a finite $K$-algebra. But then by semi continuity also the special fiber will not be finite.
    \hfill $\Box$
\end{proof}
\subsection{Boundedness for stable $\cG$-torsors}
With the canonical reduction at hand we can now deduce:

\begin{proposition}\label{prop:open}
    Let $\cL_{\det,\un{\chi}}$ be an admissible line bundle on $\Bun_{\cG}$.
    Then the stacks  $\Bun_{\cG}^{st} \subset \Bun_{\cG}^{sst} \subset \Bun_{\cG}$ of $\cL_{\det,\un{\chi}}$-(semi)-stable $\cG$-torsors are open substacks of finite type.
\end{proposition}
\begin{proof}
    By Lemma \ref{lem:semicont} instability and strict semistability are stable under specialization. Therefore we only need to show that $\Bun_{\cG}^{sst}$ is constructible and contained in a substack of finite type.
    Again we argue as in \cite{Behrend}. First we show that for any $c\geq 0$ the stack of $\cG$-torsors of fixed degree satisfying $\mu_{\max}(\cE) \leq c$ is of finite type. To prove this we may suppose that $\chi=0$, as the linear function $\lambda \mapsto \langle \chi,\lambda \rangle$ only changes $\frac{\wt(\lambda)}{||\lambda||}$ by a finite constant.

    By \cite{Behrend} the claim holds if $\cG=\cG_0$ is a reductive group scheme over $C$. Moreover if $\cG^\prime \to \cG_0$ is a parahoric group scheme mapping to $\cG_0$ such that $\cG^\prime$ is an Iwahori group scheme at all $x\in \Ram(\cG^\prime)$ then the morphism $\Bun_{\cG^\prime} \to \Bun_{\cG_0}$ is a smooth morphism with fibers isomorphic to a product of flag varieties $\cG_{0,x}/\cB_x$. Moreover the cokernel of $\Lie(\cG^\prime) \to \Lie(\cG_0)$ is of finite length. Thus there exists a constant $d$ such that for any $\cE$ is a $\cG^\prime$ torsor we have that if  $\cE\times^{\cG^\prime}\cG_0$ admits a reduction of slope $\mu(\cE\times^{\cG^\prime}\cG_0) > c+d$ then this reduction induces a reduction of $\cE$ of slope $\mu(\cE) >c$.

    Therefore the result also holds for $\cG^\prime$. Now any parahoric group scheme contains an Iwahori group scheme so that by the same reasoning the result also holds if $\cG$ is any parahoric Bruhat--Tits group scheme such that $\cG|_{C-\Ram(\cG)}$ admits an unramified extension.

    Now choose $\pi\colon \tilde{C} \to C$ a finite Galois covering with group $\Gamma$ such that $\pi^*\cG|_{C-\Ram(\cG)}$ is a generically split reductive group scheme on $\tilde{C}-\pi^{-1}(\Ram(\cG))$, equipped with a $\Gamma$-action. We can extend this group scheme to a Bruhat--Tits group scheme $\tilde{\cG}$ that is $\Gamma$ equivariant and admits an morphism $\pi^*(\cG) \to \tilde{\cG}$ that is an isomorphism over $\tilde{C}-\pi^{-1}(\Ram(\cG))$. We already know the result for $\Bun_{\tilde{\cG}}$. Now if $\cE$ is a $\cG$-torsor then $\tilde{\cE}:= \pi^*{\cE} \times^{\pi^* \cG} \tilde{\cG}$ is a $\Gamma$-equivariant $\tilde{\cG}$ torsor. Now if $\tilde{\cE}$ admits a canonical reduction to a parabolic subgroup $\tilde{\cP}$ this will define an equivariant reduction of $\pi^*\cE|_{C-\Ram(\cG)}$ and therefore a reduction of $\cE$.

    Again we can compare the weights of the reductions because
    $\pi_* \ad(\pi^*(\cE)) = \ad(\cE) \tensor \pi_*(\cO_{\tilde{C}})$.
    Therefore the determinant of the cohomology  $\det(H^*(\tilde{C}, \ad(\pi^* \cE)) = \det H^*( C ,\ad(\cE) \tensor \pi_* \cO_{\tilde{C}})$ defines a power of $\cL_{\det}$ on $\Bun_{\cG}$. This implies that the weight of the reduction of $\pi^*\cE$ is a just  $\deg(\pi)$-times the weight of the induced reduction of $\cE$. So again a very destabilizing canonical reduction of $\tilde{\cE}$ induces a destabilizing reduction of $\cE$.

    We are left to show that the stable and semistable loci are constructible. We saw that unstable bundles admit a canonical reduction to some $\cP\subset \cG$, i.e., they lie in the image of the natural morphism $\Bun_\cP \to \Bun_{\cG}$.
    From Remark \ref{Rem:GloballyFinitelyManyConjClassesOfP} we know that it suffices to consider the image of this morphism for finitely many $\cP$.

    To conclude, we need to see that for every substack of finite type $U \subset \Bun_{\cG}$ only finitely many connected components of $\Bun_{\cP}$ contain canonical reductions that map to $U$. We just proved that on $U$ the slope $\mu_{\max}$ is bounded above.

    We claim that for the canonical reduction of a $\cG$-bundle with $\mu_{\max}\leq c$ the degree of the corresponding $\cP$-bundle $\un{\deg} \in (X^*(\cP))^\vee$ lies in a finite set. By construction the canonical reduction was obtained from a complementary polyhedron and therefore the $\cL$-degree of the reduction that was defined from the weight of the reduction that is bounded below when evaluated at fundamental weights. As we bounded the $\mu_{max}$ this will also be bounded above.
    Now we saw that the from equation \ref{wtEquation} that the weight of the reduction can be computed from the degree of the reduction and a local term only depending on the group $\cP$. Therefore we find obtain our bound for the degree of the $\cP$-bundle.

    To conclude we use that the number of connected components of $\Bun_{\cP}$ of a fixed degree is finite. This is known for $\cG$-bundles (see Proposition \ref{Prop:pi0BunG}  ). From this we can deduce our statement by looking at the Levi quotient $\cP \to \cL$ and the morphism $\Bun_{\cP} \to \Bun_{\cL}$, which is smooth with connected fibers, because the kernel $\cU\to \cP$ has a filtration by additive groups.

    For the stack of stable points we can argue in the same way, oserving that strictly semistable bundles admit a reduction to a parabolic with maximal slope equal to $0$.
    \hfill $\Box$
\end{proof}

\subsection{Conclusion for $\cG$-torsors}

\begin{theorem}\label{Thm:parahoric_moduli}
    Let $\cG$ be a parahoric Bruhat--Tits group scheme that splits over a tamely ramified extension $\tilde{k(C)}/k(C)$ and let $\cL=\cL_{\det,\un{\chi}}$ be an admissible line bundle on $\Bun_{\cG}$. Then the stack of $\cL$-stable $\cG$ bundles $\Bun_{\cG}^{st}$ admits a separated coarse moduli space of finite type over $k$.
\end{theorem}

\begin{proof}
This now follows from Proposition \ref{Prop:coarse}, because the
stack of stable $\cG$-torsors $\Bun_{\cG}^{st}$ is an open substack
by Proposition \ref{prop:open} and it satisfies the conditions of
Proposition \ref{Prop:sep} by Proposition \ref{Prop:BuncGstar}.
\hfill $\Box$
\end{proof}

\begin{remark}{\rm
    For admissible line bundles $\cL=\cL_{\det,\un{\chi}}$ the results on the existence of canonical reductions for $\cG$-bundles allow  to copy the proof of the semistable reduction theorem using Langton's algorithm from \cite{SemistableReduction} and \cite{Addendum}. In particular in those cases where $\cL$-semistability is equivalent to $\cL$-stability this then implies that the coarse moduli spaces are proper.}
\end{remark}

\section{Appendix: Fixing notations for Bruhat--Tits group schemes}\label{AppendixBT}
In this appendix we collect the results from \cite{BT2} on the
structure of Bruhat--Tits group schemes that we use. As the
definition is local let us fix a discrete valuation ring $R$ with
fraction field $K$ and residue field $k$ and $G_K$ a reductive group
over $K$. In general the groups are defined by descent starting from
the quasi-split case over an unramified extension of $R$. In our
applications we can always extend the base field and if $k=\kbar$
the group $G_K$ is quasi-split by the theorem of Steinberg
\cite[Section 8.6]{BorelSpringer}. We will therefore assume that
$G_K$ is quasi-split.

We choose a maximal split torus $S_K \subset G_K$ and denote
$T_K:=Z(S_K)$ the centralizer of $S$ which is a maximal torus of
$G_K$ because $G_K$ was quasi-split.\footnote{In \cite{BS} Balaji
and Seshadri study the case where $S_K=T_K$ is a split maximal
torus, i.e. the case where $G_K$ is a split reductive group, which
already shows many interesting features.}

To construct models of $G_K$ over $R$ Bruhat--Tits first extend the
torus $T_K$ to a scheme over $R$ and then the root subgroups of
$G_K$ using a pinning of $G_K$ that they upgrade to a
Chevalley-Steinberg valuation $\varphi$. (\cite{BT2} Section 4.2.1
and 4.1.3). Let us recall these notions:

\subsection{Chevalley--Steinberg systems, pinnings and valuations}
If $G_K$ is split, a Chevalley--Steinberg system is simply a pinning
of our group, i.e., an identification $x_\alpha\colon \bG_a
\map{\cong} U_\alpha$ which is compatible for $\alpha,-\alpha$ in
the sense that it comes from an embedding of $\zeta_\alpha \colon
\SL_{2,K} \to G_K$ identifying $\bG_a$ with the strict upper and
lower triangular matrices \cite[Paragraph 3.2.1]{BT2} and is
compatible with reflections \cite[Paragraph 3.2.2,
(Ch1),(Ch2)]{BT2}.

If $G_K$ is not split, we can split it over a Galois extension
$\tilde{K}$ and choose an equivariant pinning \cite[Paragraph
4.1.3]{BT2}: Let $\Gamma:=\Gal(\tilde{K}/K)$ and denote by
$\tilde{\Phi}$ the roots with respect to $T_{\tilde{K}}$ of
$G_{\tilde{K}}$ and $\Phi$ the roots for $S$. Then $X^*(S) =
X^*(S_{\tilde{K}})$ and the elements of $\Phi$ are the restrictions
of elements of $\tilde{\Phi}$ to $S$.

The root subgroups $U_a$ can then be described as follows. For any
root ray $a\in \Phi$ denote $\tilde{\Delta}_a\subset \tilde{\Phi}$
the set of simple roots that restrict to $a$. The analog of the
$\SL_2$ defined by a root is a morphism $\zeta_a\colon G^a_K \to
G_K$. If $a$ is not a multiple root we have $$G^a_K \cong
\Res_{L_\alpha/K} \SL_2,$$ where $L_\alpha\subset \tilde{K}$ is the
field obtained by the stabilizer of any $\alpha \in
\tilde{\Delta}_a$. In this case a pinning is an isomorphism
$$\Res_{L_\alpha/K} \bG_a \map{\cong} U_a,$$ which is again assumed
to be compatible for $a$ and $-a$.  (\ \cite[Paragraph
4.1.7,4.1.8]{BT2})

If $a$ is a multiple root ray $G^a$ is the Weil restriction of a
unitary group\footnote{This happens if the Galois group interchanges
neighboring roots in the Dynkin diagram.}. In this case for any pair
$\alpha,\alpha^\prime \in \tilde{\Delta}_a$ such that
$\alpha+\alpha^\prime$ is a root let $L_\alpha =
\tilde{K}^{\Stab{\alpha}}$ which is a quadratic extension of
$L_{\alpha+\alpha^\prime} =:L_2$. This extension defines the unitary
group $SU_3(L_\alpha/L_2))$ over $L_2$ (with respect to the standard
hermition form). Then $$G^a = \Res_{L_2/K} SU_3(L_\alpha/L_2).$$

In this case the root subgroups $U_a,U_{-a}$ are of the form
$$ U_a(L_2) =\Big\{x_a(u,v)=\left(\begin{array}{ccc}
1 &-u^\sigma & -v \\ 0 & 1 & u \\ 0 & 0 & 1
\end{array}\right)\Big\},\quad U_{-a}(L_2) =\Big\{ x_{-a}(u,v)=\left(\begin{array}{ccc}
1 & 0 & 0 \\ u & 1 & 0 \\ -v & -u^\sigma &1
\end{array}\right)\Big\}$$
with $v+v^\sigma = uu^\sigma.$ This has a filtration $U_{2a} \cong
\{ v \in L_{\alpha}/L_2 | \tr(v) =0 \}$ and $U_a/U_{2a} \cong
L_{\alpha}$.

The Chevalley pinning induces valuations on the groups $U_a$. For
non multipliable roots one sets  $\phi_\alpha(x_\alpha(u)):=|u|$ and
for multiple roots one defines $\phi_a(x_a(u,v)):=\1halb |v|$ and
$\phi_{2a}(x_a(0,v)):= |v|$.

(These choices define a valued root system and identify the standard
appartment $A\cong X_*(S)_{\bR}$ in the Bruhat--Tits building of
$G$).

\subsection{Parahoric group schemes}

With this notation, we can recall the construction of an open part
of parhoric group schemes.

For the torus $T_K$ Bruhat and Tits choose a version of the N\'eron
model as extension. In order to be consistent with the conventions
from \cite{PR} we choose the {\em connected} N\'eron model $\cT_R$
as an extension over $\Spec R$.

For the unipotent groups $U_a$, the valuations introduced above can
be used to define extensions of $U_a$ to group schemes $\cU_{a,k}$
over $\Spec R$ for any $k\in \bR$. For non-multiple roots the
pinning identifies the abstract group $U_{a,k} := \{ u\in U_a(K) |
\phi_a(u) \leq k\} \cong \{ u\in L_a| |u|\leq k\}$ and this can be
equipped with the structure of a groups scheme isomorphic to
$\Res_{R_a/R} \bG_a$.

For multiple roots this is slightly more complicated to spell out,
but again these group schemes always correspond to free $R$ modules
\cite[4.3.9]{BT2}.

The open subset of a parahoric group scheme will be of the form
(\cite[Theorem 3.8.1]{BT2}):
$$ \prod_{a\in \Phi^-} \cU_{a,f(a)} \times \cT \times \prod_{a\in \Phi^+} \cU_{a,f(a)}.$$

Now a facet $\Omega\subset A$ defines a valuation of the root system
(4.6.26)
$$f(a):= \inf \{k\in \bR | a(x)+k \geq 0 \forall a \in \Omega \},$$
here we used our Chevalley-Steinberg valuation, which defines an
isomorphism $-\phi\colon A\cong X_*(S)_\bR$.

The product described above carries a birational group law
(\cite[Propostion 5.12]{Landvogt}) and thus one can use
\cite[Theorem 5]{BLR} to construct a the group scheme $\cG_{\Omega}$
(denoted by $\cG_{\Omega}^\circ$ in \cite{BT2}), containing the
product as an open neighborhood of the identity. For our
computations this is sufficient as these only use the Lie-algebra of
$\cG$.

\section{Appendix: Basic results on $\Bun_\cG$}

As in Section \ref{Sec:Parahoric} we fix a smooth projective
geometrically connected curve $C$ over an algebraically closed field
$k$ and $\cG\to C$ a parahoric Bruhat-Tits group scheme. In this
appendix we collect some results on the stack of $\cG$-bundles
$\Bun_\cG$ for which we could not find a reference.

The basic tool will be the Beilinson--Drinfeld Grassmannian
$\GR_\cG$ i.e., the ind-projective scheme representing the functor
of $\cG$-bundles together with a trivialization outside of a finite
divisor on $C$:
$$\GR_{\cG}(T) = \left\{ (\cE,D,\phi) \left| { \cE\in \Bun_{\cG}(T), D\in C^{(d)}(T) \text{ for some d} \atop \phi \colon \cE|_{C\times T - D} \map{\cong} \cG \times_C (C\times T - D) }\right.\right\}.$$

It comes equipped with a natural forgetful maps
$$ p\colon \GR_{\cG} \to \Bun_{\cG},$$
$$ \supp\colon \GR_{\cG} \to \coprod_d C^{(d)}.$$

For Bruhat-Tits groups it will be useful to consider the open
subfunctor, where the divisor $D$ does not intesect some fixed
finite subset of $C$, e.g., the ramification points of the group
$\cG$. Given $S\subset C$ finite, we will denote by
$\GR_{\cG,C-S}:=\supp^{-1} \big(\coprod (C-S)^{(d)}\big) $, i.e. the
ind-scheme parametrizing $\cG$ bundles together with a
trivialization outside of a divisor $D$ that is disjoint form $S$.
Similarly we denote by $\Gr_{\cG,x}:=\supp^{-1}(x)$, the space
parametrizing $\cG$-bundles equipped with a trivialization on
$C-\{x\}$.

The following is a probably well-known geometric version of a weak
approximation theorem:
\begin{proposition}\label{Prop:WeakApproximiation}
    For any Bruhat-Tits scheme $\cG \to C$ and any finite subscheme $S\subset C$ the natural forgetful map $p_{C-S}\colon \GR_{\cG,C-S} \to \Bun_\cG$ is surjective, i.e., after possibly passing to a flat extension, any $\cG$-bundle admits a trivialization on an open neighborhood of $S$.
\end{proposition}
\begin{proof}
    For $S=\emptyset$ this follows from a theorem of Steinberg and Borel-Springer \cite{BorelSpringer}, stating that for any algebraically closed field $\overline{K}$ any $\cG$-bundle on $C_{\overline{K}}$ is trivial over the generic point $\overline{K}(C)$. As any such trivialization is defined over an open subset this allows to deduce that $p\colon \GR_{\cG} \to \Bun_{\cG}$ is surjective in the fppf-topology.

    To deduce the result for $p_{C-S}$ we can argue as in \cite[Theorem 5]{Unif}:
    As $\GR_{\cG,C-S} \subset \GR_{\cG} $ is a dense open subfunctor, it follows that the generic point of any connected component of $\Bun_{\cG}$ is in the image of $p_{C-S}$.
    Let $\cE\in \Bun_{\cG}$ be any bundle. For the inner form $\cG^{\cE}:=\Aut_{\cG}(\cE)$ of $\cG$ we can apply the same reasoning and find that the image of the map $\GR_{\cG^{\cE},C-S} \to \Bun_{\cG^\cE} \cong \Bun_\cG$ also contains an open subset of every connected component. In particular there exist $\cG$-bundles $\cE^\prime$ in the image that are also in the image of $p_{C-S}$. For such any such  bundle $\cE^\prime$ there exist divisors $D_1,D_2$ on $C-S$ such that $\cE^\prime|_{C-D_1} \cong \cG|_{C-D_1}$ and $\cE^\prime|_{C-D_2} \cong \cE|_{C-D_2}$. Composing these isomorphisms we find that $\cE|_{C-(D_1\cup D_2)} \cong \cG|_{C-(D_1 \cup D_2)}$, i.e., $\cE$ is also in the image of $p_{C-S}$. Thus $p_{C-S}$ is surjective on geometric points. As it is the restriction of a flat morphism to an open subfunctor this implies thst it is again fppf surjective.
    \hfill $\Box$
\end{proof}
To formulate properties of the set of connected components of
$\Bun_{\cG}$ we will need some more notation. The generic point of
$C$ will be denoted by $\eta=\Spec k(C)$ and $\overline{\eta}$ will
be a geometric generic point. From \cite[Theorem 0.1]{PR} we know
that for any $x\in C$ there is a natural isomorphism
$\pi_0(\Gr_{\cG,x}) \cong \pi_1(\cG_{\overline{\eta}})_I$ where
$I=\Gal({\overline{K}_x/K_x})$ is the inertia group at $x$ and
$\pi_1(\cG_{\overline{\eta}})$ is the algebraic fundamental group,
i.e., the quotient of the coweight lattice by the coroot lattice of
$\cG_{\overline{\eta}}$.

We denote by $X^*(\cG)=\Hom(\cG,\bG_{m,C})$ the group of characters
of $\cG$. As $\cG$ is a smooth group scheme with connected fibers
$X^*(\cG) \cong X^*(\cG_\eta) =
X^*(\cG_{\overline{\eta}})^{\Gal(\overline{\eta}/\eta)}$. As any
character $\chi\colon \cG \to \bG_m$ defines a morphism $\Bun_{\cG}
\to \Pic_X$ given by $\cE \mapsto \cE(\chi):=\cE\times^{\cG}\bG_m$
it defines a degree $\un{d}\colon \Bun_{\cG} \to X^*(\cG)^\vee$ by
$\cE \mapsto \deg(\cE(\chi))$ and we will denote by
$\Bun_{\cG}^{\un{d}} \subset \Bun_{\cG}$ the substack of bundles of
degree $\un{d}$, which is open and closed because the degree of line
bundles is locally constant in families.

\begin{proposition}\label{Prop:pi0BunG}
    \begin{enumerate}
        \item[]
        \item[\rm (1)] The natural map $\pi_1(\cG_{\overline{\eta}})_{\Gal(\overline{\eta}/\eta)} \to \pi_0(\Bun_\cG)$ is surjective.
        \item[\rm (2)] For any $\un{d}\in \Hom(X^*(\cG),\bZ)$ the stack $\Bun_{\cG}^{\un{d}}$ has finitely many connected components.
    \end{enumerate}
\end{proposition}
\begin{proof}
    The first part follows from the surjectivity of $\GR_\cG \to \Bun_{\cG}$ and the description of the connected components $\pi_0(\Gr_{\cG,x}) \cong \pi_1(\cG_{\overline{\eta}})_I$ from \cite{PR} as follows. Let $\GR^d_\cG$ denote the components of $\GR_{\cG}$ mapping to $C^{(d)}$. As the projection $\GR^d_\cG$ is ind-projective every connected component intersects the fibers over the diagonal $C \subset C^{(d)}$. The preimage of the diagonal is isomorphic to $\GR_{\cG}^1$. For $\GR_{\cG}^1$ the fiber wise isomorphism   $\pi_0(\Gr_{\cG,x}) \cong \pi_1(\cG_{\overline{\eta}})_I$ is given by the Kottwitz homomorphism which by construction is induced from a Galois-equivariant  map $\pi_1(\cG_{\overline{\eta}})$ to the sheaf of connected components of the fibers of $p$. Thus this induces a surjection $\pi_1(\cG_{\overline{\eta}})_{\Gal(\overline{\eta}/\eta)} \to \pi_0(\GR_\cG^1) \to \pi_0(\Bun_{\cG})$, which proves (1).

    This implies (2), because the map $ \Hom(X^*(\cG),\bZ)^I \to \Hom(X^*(\cG),\bZ)_I$ induces an isomorphism up to torsion.
    \hfill $\Box$
\end{proof}
%

\providecommand{\bysame}{\leavevmode\hbox to3em{\hrulefill}\thinspace}
%
%

\bibliographystyle{amsalpha}
\bibliographymark{References}

\end{document}